\documentclass[reqno]{amsart}

\usepackage[sorting=none, sorting=nyt, maxbibnames=99, backend=biber]{biblatex}

\renewbibmacro{in:}{}
\bibliography{references.bib}
\usepackage{amssymb}
\usepackage{mathrsfs}
\usepackage{graphics,graphicx, mathrsfs}
\usepackage{enumerate}
\usepackage{enumitem}
\usepackage{textcomp}
\usepackage{url}
\usepackage{xcolor}
\definecolor{vio}{rgb}{0.54, 0.17, 0.89}
\usepackage[ruled, lined, linesnumbered, longend]{algorithm2e}
\usepackage{hyperref}
\usepackage{titlesec}

\newtheorem{theorem}{Theorem}[section]
\newtheorem{lemma}[theorem]{Lemma}

\newtheorem{proposition}[theorem]{Proposition}

\numberwithin{equation}{section}

\theoremstyle{remark}
\newtheorem*{remark}{Remark}

\theoremstyle{definition}
\newtheorem{definition}{Definition}

\titleformat{\section}
  {\normalfont\large\bfseries\centering}{\thesection}{1em}{}

\titleformat{\subsection}
  {\normalfont\bfseries}{\thesubsection}{1em}{}

\DeclareMathOperator{\li}{\mathrm{li}}

\DeclareMathOperator{\mcA}{\mathcal{A}}
\DeclareMathOperator{\mcX}{\mathcal{X}}
\DeclareMathOperator{\PP}{\mathcal{P}}

\def\reals{\hbox{\rm I\kern-.18em R}}
\def\complexes{\hbox{\rm C\kern-.43em
\vrule depth 0ex height 1.4ex width .05em\kern.41em}}
\def\field{\hbox{\rm I\kern-.18em F}} %symbol for field

\let\svthefootnote\thefootnote
\newcommand\freefootnote[1]{%
  \let\thefootnote\relax%
  \footnotetext{#1}%
  \let\thefootnote\svthefootnote%
}

\newenvironment{section*}[2][A]{
  \section*{#2}
  \renewcommand\thesection{#1}
  \setcounter{theorem}{0}}{}

\allowdisplaybreaks

\begin{document}

\title[An effective Bombieri--Vinogradov error term for sifting problems]{An effective Bombieri--Vinogradov error term for sifting problems}

\author{Daniel R. Johnston}
\address{Max Planck Institute for Mathematics}
\email{johnston@mpim-bonn.mpg.de}
\thanks{This research was partially supported by Australian Research Council Discovery Project DP240100186 and an Australian Mathematical Society Lift-off
Fellowship.}
\date\today

\begin{abstract}
    In number theory, many major results related to the additive properties of primes are proven using the methods of sieve theory. However, in nearly every case, the existing proofs of these results are ineffective, in that explicit values for which they hold cannot be computed. The reason for this ineffectivity is the reliance on the Bombieri--Vinogradov theorem. In this paper, we show that any classical sifting problem with a Bombieri--Vinogradov style error term can in fact be made effective, with no loss to the asymptotic form of the original (ineffective) result. This is done by carefully modifying the sieve upper and lower bounds to avoid the usual complications regarding Siegel zeros. We also provide some simple applications. For example, we show that one may effectively bound the number of primes $p\leq x$ such that $p+2$ is also prime by
    \begin{equation*}
        (4+o(1))C_2\frac{x}{(\log x)^2},
    \end{equation*}
    where
    \begin{equation*}
         C_2=2\prod_{p>2}\left(1-\frac{1}{(p-1)^2}\right).
    \end{equation*}
\end{abstract}

\maketitle

\freefootnote{\textit{Affiliation}: Max Planck Institute for Mathematics, Bonn, Germany.}
\freefootnote{\textit{Key phrases}: Bombieri--Vinogradov theorem, sieve methods, effective results, twin primes, Goldbach representations.}
\freefootnote{\textit{2020 Mathematics Subject Classification}: 11N35 (primary) 11M20, 11N36 (Secondary)}

\section{Introduction}
In this paper we are interested in the Bombieri--Vinogradov theorem and its application to sifting problems. The Bombieri--Vinogradov theorem is a central tool in analytic number theory given as follows.
\begin{theorem}[Bombieri--Vinogradov]\label{bomthm}
    For any $A>0$,
    \begin{equation}\label{bomvineq}
        \sum_{d\leq D}\sup_{y\leq x}\max_{(a,d)=1}\left|\pi(x;d,a)-\frac{\pi(x)}{\varphi(d)}\right|=O_A\left(\frac{x}{(\log x)^A}\right),
    \end{equation}
    where $D=\sqrt{x}/(\log x)^B$ for some constant $B>0$ depending on $A$.
\end{theorem}
Here, as usual, $\pi(x)$ is the number of primes less than or equal to $x$, 
\begin{equation*}
    \pi(x;d,a)=\#\{p\leq x:p\equiv a\ \text{(mod $d$)}\}
\end{equation*}
counts the number of primes in the arithmetic progression $a$ mod $d$, and $\varphi(\cdot)$ is the Euler totient function. 

Up to a power of $\log x$,~\eqref{bomvineq} is the same bound as that obtained by assuming the Generalised Riemann Hypothesis, in which one has the pointwise bound (see e.g.\ \cite[Corollary 13.8]{montgomery2007multiplicative})
\begin{equation*}
    \left|\pi(x;d,a)-\frac{\pi(x)}{\varphi(d)}\right|=O(\sqrt{x}\log x).
\end{equation*}

One of the major shortcomings of the Bombieri--Vinogradov theorem is that every known proof is ineffective. In particular, it is not possible to explicitly determine how large $x$ must be in terms of $A$ to obtain a bound of the strength \eqref{bomvineq}. This means that any direct application of the Bombieri--Vinogradov theorem is also ineffective. The main reason for this ineffectivity is the potential existence of Siegel zeros, that is, real-valued zeros of Dirichlet $L$-functions very close to $1$.

In this paper we build upon recent work (e.g.~\cite{akbary2015variant,Yamada2015,liu2017effective,BJV2025}) on developing effective variants of the Bombieri--Vinogradov theorem, which circumvent the usual complications regarding Siegel zeros. Our main results (Theorems \ref{upperthm} and \ref{lowerthm}) are in the context of sieve theory and thus yield effective versions of many bounds, including those related to the twin prime and Goldbach conjectures. To make these sifting results effective, we have to carefully treat the case of ``moderate" moduli (approximately $\log x\leq d\leq (\log x)^A$) in the Bombieri--Vinogradov theorem under the assumption of Siegel zeros. This is done by working with slightly altered sifting sets, which avoid any Siegel zeros and introduce an asymptotically small error. By contrast, in classical treatments this range of moduli is covered using Siegel's ineffective bound for real zeros of Dirichlet $L$-functions~\cite{siegel1935klassenzahl}.

As a simple application, we prove the following effective upper bounds. See also Theorems \ref{BJVthm} and \ref{p2thm} for examples of effective lower bound results that can be obtained via our methods. 
\begin{theorem}\label{twingoldthm}
    For $x>0$ and $n$ even, let
    \begin{align}
        \Pi_2(x)&=\#\{p\leq x:p\ \text{prime}\ \text{and}\ p+2\ \text{prime}\},\ \text{and}\label{pi2def}\\
        G(n)&=\#\{p\leq n: n-p\ \text{is}\ \text{prime}\}\label{Gndef}
    \end{align}
    count the number of twin primes less than $x$ and the number of representations of $n$ as the sum of two primes. Further, let
    \begin{equation*}
        C_2=2\prod_{p>2}\left(1-\frac{1}{(p-1)^2}\right)\quad\text{and}\quad C_n=C_2\prod_{\substack{p\mid n\\p>2}}\frac{p-1}{p-2}.
    \end{equation*}
    Then, for any $\varepsilon>0$ and sufficiently large $x\geq X(\varepsilon)$ and $n\geq N(\varepsilon)$, one has
    \begin{equation}\label{pi2gbound}
        \Pi_2(x)\leq (4+\varepsilon)C_2\frac{x}{(\log x)^2}\quad\text{and}\quad G(n)\leq(4+\varepsilon)C_n\frac{n}{(\log n)^2}
    \end{equation}
    where $X(\varepsilon)$ and $N(\varepsilon)$ are effectively computable.
\end{theorem}
Note that heuristics for the twin prime and Goldbach conjectures suggest that (see \cite[Conjectures A and B]{hardy1923some})
\begin{equation*}
    \Pi_2(x)\sim C_2\frac{x}{(\log x)^2}\quad\text{and}\quad G(n)\sim C_n\frac{n}{(\log n)^2}.
\end{equation*}
Previously, the best effective estimates for $\Pi_2(x)$ and $G(n)$ gave a constant of $8$ in place of $4+\varepsilon$ in \eqref{pi2gbound} (\cite{siebert1976montgomery} or \cite[Lemma~5]{riesel1983sums}). Thus, Theorem \ref{twingoldthm} can be viewed as improving upon previous work by a factor of 2.

If one is not concerned with effectivity, then the constant $4+\varepsilon$ can be improved with more elaborate techniques. Most notably, in recent work by Lichtman \cite{lichtman2023primes,lichtman2025modification} and Pascadi \cite{pascadi2025exponents}, improved forms for the remainder term in the linear sieve are given, which allow one to go further than is possible with the Bombieri--Vinogradov theorem. This yields leading constants of $3.203+\varepsilon$ for $\Pi_2(x)$ \cite[Corollary 1.4]{pascadi2025exponents} and $3.3907+\varepsilon$ for $G(n)$ \cite[Theorem 1.2]{lichtman2023primes}. It is likely that some variant of our results could be proven that makes Lichtman and Pascadi's results effective, but we do not attempt to do so here.

Finally, we remark that one can readily obtain strong bounds of the form~\eqref{pi2gbound} under the assumption of well-known conjectures. For example, assuming an effective form of the Elliott-Halberstam conjecture (which states that~\eqref{bomvineq} holds for $D$ as large as $x^{1-\varepsilon}$) gives~\eqref{pi2gbound} with a leading constant of $2+\varepsilon$.

\subsection{Overview of related results}
In all standard proofs of the Bombieri--Vinogradov theorem, one applies the ineffective Siegel--Walfisz theorem \cite[Corollary 11.21]{montgomery2007multiplicative} for small moduli $d$; that is $d<(\log x)^B$ for any fixed $B>0$. However, if one is okay with a weaker bound than that in \eqref{bomvineq}, then it is possible to apply Page's effective bound \cite[Theorem II]{page1935number} on Siegel zeros, in place of the Siegel--Walfisz theorem. This procedure was detailed by Liu, who obtained the following result \cite[Theorem~1]{liu2017effective}.
\begin{theorem}[Liu]\label{liuthm}
    There exists a computable constant $B>0$ such that
    \begin{equation}\label{liueq}
        \sum_{d\leq D}\sup_{y\leq x}\max_{(a,d)=1}\left|\pi(y;d,a)-\frac{\pi(y)}{\varphi(d)}\right|=O\left(\frac{x(\log\log x)^9}{(\log x)^2}\right),
    \end{equation}
    where $D=\sqrt{x}/(\log x)^B$ and the big-$O$ term is effective.
\end{theorem}
\begin{remark}
    Liu actually uses the logarithmic integral $\li(x)$ in place of $\pi(x)$, but it is readily converted to the form \eqref{liueq} above by applying an effective version of the prime number theorem $\pi(x)\sim \li(x)$ such as \cite[Corollary 1.3]{johnston2023some}.
\end{remark}
Liu thus gives a fully effective version of the Bombieri--Vinogradov theorem, albeit with a restriction on the power of $\log$ in the big-$O$ term. In \cite{liu2017effective}, Liu gives several applications of Theorem \ref{liuthm}. For example, they remark that \eqref{liueq} can be inputted into the classical proofs of Hooley \cite{hooley1957representation} or Linnik \cite{linnik1960asymptotic} to determine an explicit constant $C$ such that every integer $N>C$ can be expressed as the sum of two squares and a prime.

Another effective variant of the Bombieri--Vinogradov theorem is obtained by simply excluding the small moduli $d$ from the sum in \eqref{bomvineq} that are usually covered by the Siegel--Walfisz theorem. In this direction, Akbary and Hambrook \cite{akbary2015variant} proved the following explicit result, which has since been refined by Sedunova \cite{sedunova2018partial}.
\begin{theorem}[{\cite[Corollary 1.4]{akbary2015variant}}]\label{akhamthm}
    Let $x\geq 4$ and $1\leq D_1\leq D\leq x^{1/2}$. Let $\ell(d)$ denote the least prime divisor of $d$. Then,
    \begin{align}\label{akhameq}
        &\sum_{\substack{d\leq D\\\ell(d)>D_1}}\sup_{y\leq x}\max_{(a,d)=1}\left|\pi(y;d,a)-\frac{\pi(y)}{\varphi(d)}\right|\notag\\
        &\qquad\qquad<347\left(4\frac{x}{D_1}+4\sqrt{x}D+18x^{\frac{2}{3}}D^{\frac{1}{2}}+5x^{\frac{5}{6}}\log\left(\frac{eD}{D_1}\right)\right)(\log x)^{\frac{9}{2}}.
    \end{align}
\end{theorem}
In particular, if $D_1$ is a suitably large power of $\log x$, then \eqref{akhameq} reduces to the Bombieri--Vinogradov theorem with a restricted sum over $d$. In \cite[Section~2]{akbary2015variant}, Akbary and Hambrook similarly list applications of Theorem \ref{akhamthm}, whereby the Bombieri--Vinogradov theorem is only required in the weaker form \eqref{akhameq}. For example, in \cite[Theorem 2.2]{akbary2015variant} they give a fully explicit bound for the partial sums
\begin{equation}\label{normalpm1}
    \sum_{p\leq x}\left(\omega(p-1)-\log\log x\right)^2
\end{equation}
where $\omega(p-1)$ is the number of distinct prime divisors of $p-1$ with $p$ prime. Notably, \eqref{normalpm1} is closely related to the Titchmarsh divisor problem and a classical result of Erd\H{o}s \cite{erdos1935normal} on the normal order of $\omega(p-1)$.

\subsection{Sieve theory and the Bombieri--Vinogradov theorem}
A major application of the Bombieri--Vinogradov theorem is to sieve theory. In sieve theory one considers a finite set of integers $\mcA$ to be sifted by an infinite set of primes $\PP$. More precisely, one is concerned with bounding
\begin{equation}\label{SAPzeq}
    S(\mathcal{A},\mathcal{P},z):=\#\{a\in\mcA:p\in\PP\ \text{and}\ p\mid a\implies p\geq z\}.
\end{equation}
In many settings, the Bombieri--Vinogradov theorem is used to obtain upper and lower bounds for $S(\mcA,\PP,z)$. As a (non-exhaustive) list of examples, one can use these upper and lower bounds to obtain: 
\begin{enumerate}
    \item Upper bounds for the number of twin primes $\leq x$ \cite[Theorem 3.11]{halberstam1974sieve}.
    \item Upper bounds on the number of Goldbach representations of an even number $N$ as the sum of two primes \cite[Theorem 3.11]{halberstam1974sieve}.
    \item Upper bounds for the number of primes represented by a polynomial $F(p)$ with prime arguments \cite[Theorem 5.6]{halberstam1974sieve}.
    \item Lower bounds for the number of representations of even integers $N$ as the sum of a prime and a number with at most $2$ prime factors \cite[Theorem~11.1]{halberstam1974sieve}.
\end{enumerate}
See also \cite{franze2020almost,li2023additive,johnston2023some} for some generalisations of these results. Each of the applications above requires a stronger form of the Bombieri--Vinogradov theorem than that given in Theorems \ref{liuthm} and \ref{akhamthm}. Thus, further work is required to make them effective. 

In the case of the application (4) above, known as Chen's theorem, an effective and explicit version was recently obtained by Bordignon, Starichkova and the author of this paper \cite{BJV2025}. This result was obtained by significantly correcting and improving upon unpublished work of Yamada \cite{Yamada2015}.
\begin{theorem}[{\cite[Corollary 1.5]{BJV2025}}]\label{BJVthm}
    Every even integer $N>\exp(\exp(32.7))$ can be represented as the sum of a prime and a square-free number with at most two prime factors.
\end{theorem}
To avoid the ineffectivity of the Bombieri--Vinogradov theorem, the proof of Theorem \ref{BJVthm} relates the sifting sets $\mcA$ and $\PP$ to similar sets which avoid the main complications involving Siegel zeros. However, since the proof of Theorem \ref{BJVthm} also involves a delicate sieve weighting procedure and is fully explicit, the core ideas of the argument are obscured by technical computations.

The aim of this paper is to generalise the techniques used to prove Theorem \ref{BJVthm}. Namely, we prove an effective sieve upper and lower bound that works for any sifting problem with a Bombieri--Vinogradov style error term, including every example (1)--(4) listed at the start of this subsection. Compared to the work in \cite{BJV2025}, our argument is simpler and contains stronger bounds. Our approach is thus likely to yield better numerics than those in \cite{BJV2025}, especially if applied to simpler sifting problems such as Theorem \ref{twingoldthm} i.e.\ examples (1) and (2) above.

\subsection{Notation and statement of main results}
We now introduce the notation and statements of our main results. In the following section (Section~\ref{outlinesect}) we then give a short outline of the main arguments, before continuing with all the technical details in later sections. 

To begin with, we recall some standard notation used in sieve theory. So, consider a sifting problem with a finite set of integers $\mcA$ and an infinite set of primes $\PP$ such that no element of $\mcA$ is divisible by a prime in $\overline{\PP}=\{p\ \text{prime}:p\notin\PP\}$. For any square-free integer $d$ with\footnote{Here and throughout, by $(d,\overline{\mathcal{P}})=1$ we mean that $(d,p)=1$ for all $p\in\overline{\PP}$.} $(d,\overline{\mathcal{P}})=1$, let $X>1$ (independent of $d$) and $g(d)$ be a multiplicative function with
\begin{equation*}
    0<g(p)<1\ \text{for all}\ p\in\mathcal{P}
\end{equation*}
and 
\begin{equation}\label{vzdef}
    V(z)=\prod_{\substack{p<z\\p\in\PP}}\left(1-g(p)\right).
\end{equation}
We then let
\begin{equation}\label{Adrddef}
    \mathcal{A}_d:=\{a\in \mathcal{A}:d\mid a\}\quad\text{and}\quad r(d):=|\mathcal{A}_d|-g(d) X.
\end{equation}
Here, $X$ and $g(d)$ should be chosen so that
\begin{equation}\label{AdAasymeq}
    |\mathcal{A}_d|\sim g(d) X.
\end{equation}
We also denote the \emph{dimension} of the sifting problem as a constant $\kappa>0$ such that
\begin{equation}\label{omegacon}
    \prod_{\substack{z_1\leq p<z_2\\p\in\mathcal{P}}}\left(1-g(p)\right)^{-1}<\left(\frac{\log z_2}{\log z_1}\right)^{\kappa}\left\{1+\frac{L}{\log z_1}\right\}
\end{equation}
is satisfied for some effective constant $L>0$ and all $z_2>z_1\geq 2$. Provided $\kappa$ and $L$ are constant, one can prove generic sieve upper and lower bounds of the form
\begin{align}
    S(\mathcal{A},\mathcal{P},z)<XV(z)\left(F(s)+\varepsilon_1(X)\right)+R_D\label{sieveubgen},\\
    S(\mathcal{A},\mathcal{P},z)>XV(z)\left(f(s)+\varepsilon_2(X)\right)-R_D. \label{sievelbgen}
\end{align}
Here, $s=\log D/\log z\geq 1$, $F(s)-1>0$, $1-f(s)>0$, $F(s)$ is non-increasing, $f(s)$ is non-decreasing,  $\varepsilon_1(X)$ and $\varepsilon_2(X)$ are asymptotically decreasing towards $0$, and
\begin{equation}
    R_D=\sum_{\substack{d<D\\(d,\overline{\mathcal{P}})=1}}\mu^2(d)\gamma_0^{\omega(d)}|r(d)|\label{RDeq},
\end{equation} 
for some constant $\gamma_0>0$ with $\mu(d)$ denoting the M\"obius function and $\omega(d)$ denoting the number of distinct prime factors of $d$. Examples of such upper and lower bounds can be found throughout any standard text on sieve theory including \cite{halberstam1974sieve,greaves2013sieves} or \cite{friedlander2010opera} and are referred to as variants of ``the fundamental lemma of sieve theory". The main principle we wish to show is that if $r(d)$ is such that the remainder sum in \eqref{RDeq} is directly bounded by the Bombieri--Vinogradov theorem, then the bounds \eqref{sieveubgen} and \eqref{sievelbgen} can be modified to give a new remainder term $R_D$ of the same strength which avoids the ineffectivity due to potential Siegel zeros.

We begin with our effective upper bound result. Note that throughout we will often use Vinogradov's shorthand notation $f\ll g$ to mean $f=O(g)$, and $f\asymp g$ to mean that $f\ll g$ and $g\ll f$.

\begin{theorem}\label{upperthm}
    Suppose that, for fixed sieving dimension $\kappa>0$ one has a generic sieve upper bound of the form \eqref{sieveubgen} with a remainder term $R_D$ given by \eqref{RDeq} associated with a constant $\gamma_0>0$. Suppose then that we have a specific sifting problem $(\mcA,\PP)$ with constants $A\geq 1$, $\gamma_1>0$ such that for all $p\in\PP$, 
    \begin{equation}\label{paconds}
        g(p)\leq A/(p-1)
    \end{equation}
    and for some $\mathcal{X}\asymp X\log X$,
    \begin{equation}\label{rdbound}
        r(d)\ll \gamma_1^{\omega(d)}\sup_{y\leq\mcX}\max_{(a,d)=1}\left|\pi(y;d,a)-\frac{\pi(y)}{\varphi(d)}\right|+\sqrt{\mcX}\log\mcX
    \end{equation}
    with each implied constant effective. Let $\gamma=\gamma_0\gamma_1$. For a choice of $B>\gamma^2+1$, let
    \begin{equation}\label{Ddef}
        D=D(\mcX):=\frac{\sqrt{\mcX}}{(\log \mcX)^B}.
    \end{equation}
    Then, for any $s=\frac{\log D}{\log z}\geq 1$,
    \begin{equation}\label{sieveubnew}
        S(\mcA,\PP,z)<XV(z)\left(1+O_A\left(\frac{1}{\log\log X}\right)\right)\left(F(s)+\varepsilon_1(X)\right)+O_{B,\gamma}\left(\frac{X}{(\log X)^{B_\gamma-1}}\right)
    \end{equation}
    with both of the big-$O$ terms effective, and
     \begin{equation}\label{Bgamdef}
        B_{\gamma}=
        \begin{cases}
            B-2,&\text{if}\ 0<\gamma\leq 1,\\
            \frac{B-1-\gamma^2}{2},&\text{if}\ \gamma>1. 
        \end{cases}
    \end{equation}
\end{theorem}
So, from Theorem \ref{upperthm}, if one has a Bombieri--Vinogradov style sieve remainder term, then it is possible to get an effective upper bound for $S(\mcA,\PP,z)$ with little loss. In particular, as a function of $X$, the leading term in \eqref{sieveubnew} is asymptotically the same as in \eqref{sieveubgen}. Then, in the standard case $\gamma=1$, the $O_{B,\gamma}(\cdot)$ error term in \eqref{sieveubnew} is the same as what one would get by directly applying the Bombieri--Vinogradov theorem to $R_D$ in \eqref{sieveubgen}. With more work, it is likely that $B_\gamma$ could be improved to $B-1-\gamma$ for all $\gamma>0$, not just $\gamma=1$. Nevertheless, $\gamma=1$ is standard for many applications so we just use a simple but wasteful argument involving the Cauchy--Schwarz inequality for other values of $\gamma$ (see Proposition \ref{mainPNTAPprop}).

The proof of Theorem \ref{upperthm} is given in Section \ref{uppersect}, after obtaining preliminary results on primes in arithmetic progressions in Section \ref{sect2}. The proof of the application Theorem \ref{twingoldthm} then follows routinely from Theorem \ref{upperthm}. However, we still provide the main details of this proof in Section \ref{appsect}, thereby giving a clear example of how to apply Theorem~\ref{upperthm}.

We now move on to an effective sieve lower bound. Here, the conditions required are slightly more restrictive. However, these conditions are still easily satisfied by most sifting problems of interest.

\begin{theorem}\label{lowerthm}
    Suppose that one has a generic sieve upper bound \eqref{sieveubgen} and lower bound \eqref{sievelbgen} such that \eqref{omegacon} is satisfied for a fixed $\kappa,L>0$, with a remainder term $R_D$ given by \eqref{RDeq} associated with a constant $\gamma_0>0$. Suppose then that we have a specific sifting problem $(\mcA,\PP)$ with $X=|\mcA|$ and such that the condition \eqref{paconds} holds for some constant $A\geq 1$. We further suppose that for some ${\mcX\asymp {X\log X}}$ and $\gamma_1>0$,
    \begin{align}\label{sievelbmain2}
        |\mcA_{k d}|-g(d)|\mcA_{k}|\ll\gamma_1^{\omega(d)}\sup_{y\leq \mcX}\max_{(a,k d)=1}\left|\pi(y;k d,a)-\frac{\pi(y;k,a)}{\varphi(d)}\right|+\sqrt{\mcX}\log\mcX
    \end{align}
    for all square-free $k,d>0$ with $(k,d)=1$ and $(k d,\overline{\PP})=1$. Let $\gamma=\gamma_0\gamma_1$. For some choice of $B>\gamma^2+1$, let
    \begin{equation}\label{Ddef2}
        D=D(\mcX):=\frac{\sqrt{\mcX}}{(\log \mcX)^B}.
    \end{equation}
    Then, for any $s=\frac{\log D}{\log z}>1$, we have 
    \begin{align}\label{sievelbnew}
        S(\mcA,\PP,z)&>XV(z)\left(1+O_{A,\kappa,L,B,\gamma,s}\left(\frac{1}{\log\log X}\right)\right)\left(f(s-\delta)-\varepsilon_2(X)\right)\notag\\
        &\qquad\qquad\qquad\qquad\quad+O_{B,\gamma}\left(\frac{X}{(\log X)^{B_\gamma-1}}\frac{\log\log X}{\log\log\log X}\right)
    \end{align}
    with $\delta=o_{s,B}(1)$, every implied constant effective, and $B_{\gamma}$ is as in~\eqref{Bgamdef}.
\end{theorem}
Thus, provided $f(s)$ is continuous, we also have an effective sieve lower bound with no loss to the asymptotic form of \eqref{sievelbgen}. We remark that in some applications (see e.g.~\cite[Chapter~2, Example~5]{halberstam1974sieve}) one actually takes
\begin{equation}\label{Xapproxeq}
    X=|\mcA|+O\left(\frac{|\mcA|}{(\log|\mcA|)^K}\right) 
\end{equation}
for some large $K>0$. If this is the case, then the extra error term arising from~\eqref{Xapproxeq} can be absorbed into the error term of~\eqref{sievelbnew} provided $K$ is sufficiently large. 

The proof of Theorem \ref{lowerthm} is given in Section \ref{lowersect} and is more involved than that of Theorem \ref{upperthm}, requiring a careful inclusion-exclusion argument. Notably, Theorem \ref{lowerthm} requires both a generic upper and lower bound to prove \eqref{sievelbnew}. The remainder term in \eqref{sievelbnew} is also slightly weaker than that in \eqref{sieveubnew}.

A non-trivial combination of Theorem \ref{upperthm} and Theorem \ref{lowerthm} can be used to prove an effective version of Chen's theorem, as in Theorem \ref{BJVthm}. However, to demonstrate a simpler application of Theorem \ref{lowerthm}, we prove the following result in Section \ref{appsect}.

\begin{theorem}\label{p2thm}
    There exists a computable constant $N$ such that all $n>N$ with $n\equiv 0,2$ (mod 6) can be represented as
    \begin{equation*}
        n=q^2+\eta,
    \end{equation*}
    where $q$ is a prime and $\eta$ has at most $17$ prime factors.
\end{theorem}
We have chosen this application as an illustrative example since its proof requires a 2-dimensional sieve and consequently has a very different setup to the proof of Theorem \ref{twingoldthm}. Using a sieve-weighting procedure, the value of $17$ in Theorem \ref{p2thm} can be reduced further (see \cite[Theorems 1.5 and 1.6]{johnston2025sum}), but we do not do so here for brevity. 

\section{Structure of the main arguments}\label{outlinesect}
We now outline the structure of the main arguments required to prove Theorems~\ref{upperthm} and~\ref{lowerthm}. To begin with, we note that in a usual sieve-theoretic argument, one would just apply the generic upper and lower bounds~\eqref{sieveubgen} and~\eqref{sievelbgen} to bound $S(\mcA,\PP,z)$. With a Bombieri--Vinogradov error term of the form~\eqref{rdbound}, one is then led to estimate the remainder sum
\begin{equation*}
    R=\sum_{\substack{d<D\\(d,\overline{\mathcal{P}})=1}}\mu^2(d)\gamma_0^{\omega(d)}\gamma_1^{\omega(d)}\sup_{y\leq\mcX}\max_{(a,d)=1}\left|\pi(y;d,a)-\frac{\pi(y)}{\varphi(d)}\right|
\end{equation*}
with the smaller order term omitted for clarity. Normally, this would be achieved using the Cauchy--Schwarz inequality (to deal with the $\gamma_0$ and $\gamma_1$ factors) and the ineffective Bombieri--Vinogradov theorem. However, to obtain an effective result, we must carefully deal with the possibility of Siegel zeros. So, amongst all square-free $d$ up to $\approx (\log \mcX)^B$ with $(d,\overline{P})=1$ we suppose that there exists an exceptional zero (see Section~\ref{setupnotsub} for a definition) with modulus $d=k_1$. We then split into two cases; $k_1\leq \log\mcX$ and $k_1>\log\mcX$.

In the first case $k_1\leq \log\mcX$, we use a refined version of Page's theorem (Lemma~\ref{siegelboundlem}) to sufficiently bound the contribution of the exceptional zero, giving an effective bound $R\ll\mcX/(\log\mcX)^{B_\gamma}$.

In the second case $k_1>\log\mcX$, we first show that (Proposition~\ref{mainPNTAPprop})
\begin{equation}\label{examplesum1}
    R':=\sum_{\substack{1\leq d\leq D\\(d,\overline{\PP})=1\\k_1\nmid d}}\mu^2(d)\gamma^{\omega(d)}\sup_{y\leq \mcX}\max_{(a,d)=1}\left|\pi(y;d,a)-\frac{\pi(y)}{\varphi(d)}\right|\ll_{B,\gamma}\frac{\mcX}{(\log\mcX)^{B_{\gamma}}},
\end{equation}
where this bound is effective since we have removed all moduli $d$ with $k_1\mid d$. Now, to prove the effective upper bound (Theorem~\ref{upperthm}), we then note that
\begin{equation}\label{removeq1eq}
    S(\mcA,\PP,z)\leq S(\mcA,\PP\setminus\{q_1\},z),
\end{equation}
where $q_1$ is the largest prime factor of $k_1$. The sifting function $S(\mcA,\PP\setminus\{q_1\},z)$ can then be bounded using the generic sieve upper bound~\eqref{sieveubgen}, with the remainder term handled by the effective bound~\eqref{examplesum1}. To finish, we also show that there is no loss asymptotically when applying sieve bounds to either side of the inequality~\eqref{removeq1eq}. This boils down to the observation that
\begin{equation*}
    V(z)=\prod_{\substack{p<z\\p\in\PP}}\left(1-g(p)\right)\quad\text{and}\quad V_1(z):=\prod_{\substack{p<z\\p\in\PP\setminus\{q_1\}}}\left(1-g(p)\right)
\end{equation*}
are asymptotically equivalent, which follows from the condition~\eqref{paconds} and the assumption that $k_1>\log\mcX$. For further details on this last point, see Lemma~\ref{V1lem}.

For the effective lower bound, Theorem~\ref{lowerthm}, the idea is similar but the proof is more difficult since we do not have an equivalent lower bound to~\eqref{removeq1eq}. Instead we use an inclusion-exclusion identity (see Lemma~\ref{inclexcleq}):
\begin{equation}\label{inexidentity}
    S(\mcA,\PP,z)=\sum_{j=0}^{\ell-1}(-1)^j S(\mcA_{m_j},\PP_{j+1},z)+(-1)^{\ell} S(\mcA_{m_{\ell}},\PP_\ell,z),
\end{equation}
for suitably defined sets $\mcA_{m_j}$ and $P_j$, which depend on the prime factors of the exceptional modulus $k_1$. From here, we then apply the generic upper and lower sieve bounds, \eqref{sieveubgen} and~\eqref{sievelbgen} to each sifting function on the right-hand side of~\eqref{inexidentity}. Analogous to the use of~\eqref{examplesum1} for the effective upper bound, we then use another effective variant of the Bombieri--Vinogradov theorem (Proposition~\ref{pimidkprop}) to bound all of the resulting remainder sums. Collecting each term, we then carefully obtain a final bound which yields no asymptotic loss compared to the usual ineffective method for bounding of $S(\mcA,\PP,z)$.

\section{Results on primes in arithmetic progressions}\label{sect2}

In this section, we develop some results on primes in arithmetic progressions. In particular, we prove effective variants of the Bombieri--Vinogradov theorem needed for Theorems \ref{upperthm} and \ref{lowerthm}. Throughout all of the implied constants in our results are effectively computable. To emphasise this last point, we have chosen to regularly cite references where an explicit version of an intermediary result is established.

\subsection{Setup and notation}\label{setupnotsub}

To begin with, we first recall some basic facts about real zeros of Dirichlet $L$-functions and set up our corresponding notation. Here, for a Dirichlet character $\chi$, we write $L(s,\chi)$ for the associated $L$-function given by
\begin{equation*}
    L(s,\chi)=\sum_{n=1}^\infty\frac{\chi(n)}{n^s},\qquad\text{for $\Re(s)>1$},
\end{equation*}
and extended to a meromorphic function on the whole complex plane by analytic continuation. For some modulus $q$, we also write 
\begin{equation*}
    \mathcal{L}_q(s)=\prod_{\chi}L(s,\chi),
\end{equation*}
where the product is taken over all Dirichlet characters mod $q$.

For any fixed $q$, the function $\mathcal{L}_q(s)$ can have at most $1$ ``bad" real zero which is close enough to the $1$-line as to significantly worsen estimates on the number of primes in an arithmetic progression mod $q$. This is made precise by the following standard result. 

\begin{lemma}[{see \cite[Corollary 11.8]{montgomery2007multiplicative} and \cite{morrill2020elementary}}]\label{siegelzerolem}
    There exists an effectively computable constant $R_1>0$ such that $\mathcal{L}_q(s)$ has at most one real zero in the region
    \begin{equation*}
        \Re(s)>1-\frac{R_1}{\log q}.
    \end{equation*}
    If such a zero exists then the associated character $\chi$ is quadratic.
\end{lemma}

\begin{definition}
    For any modulus $q\geq 2$, if the zero in Lemma \ref{siegelzerolem} exists then it is called the \emph{Siegel zero} mod $q$.
\end{definition}

Another classical result related to Lemma~\ref{siegelzerolem}, is that amongst all moduli $q$ up to some height $Q$, there is at most one ``really bad" Siegel zero. This is useful for our purposes since the Bombieri--Vinogradov theorem is an averaging statement over a range of moduli.

\begin{lemma}[{see \cite[Lemma~8]{page1935number} or \cite[Theorem~2]{mccurley1984explicit}}]\label{exceptionalzerolem}
    For any $Q\geq 2$, the function
    \begin{equation}\label{primprodeq}
        \prod_{q\leq Q}{\prod_{\chi(q)}}^* L(s,\chi)
    \end{equation}
    has at most one real zero $\beta$ in the region
    \begin{equation*}
        \beta>1-\frac{R_3}{\log Q}
    \end{equation*}
    where the product $\prod_{\chi}^*$ is over all primitive characters mod $q$, and $R_3\leq R_1$ is an effectively computable constant (independent of $Q$) with $R_1$ as in Lemma \ref{siegelzerolem}. 
\end{lemma}

\begin{definition}
    For some choice of $Q\geq 2$, if the zero in Lemma \ref{exceptionalzerolem} exists then we call the zero as well as its corresponding modulus and character \emph{exceptional}.
\end{definition}

The final preliminary result that we require is an effective bound showing that Siegel or exceptional zeros cannot be ``too close" to the 1-line.

\begin{lemma}[{See \cite{goldfeld1975siegel} and \cite[Theorem 5.4.6]{razakarinoro2024siegel}}]\label{siegelboundlem}
    Let $q\geq 2$. If it exists, let $\beta$ denote the Siegel zero mod $q$ associated to the Dirichlet character $\chi$. Then, there exists an effectively computable constant $R_2>0$ (independent of $q$) such that
    \begin{equation}\label{pageeq}
        \beta<1-\frac{R_2}{\sqrt{q}}.
    \end{equation}
\end{lemma}
\begin{remark}
    Lemma~\ref{siegelboundlem} is a strengthening of Page's theorem~\cite[Theorem II]{page1935number}, which gives $1-\beta\gg 1/\sqrt{q}(\log q)^2$. Although Page's weaker bound would be sufficient for our purposes, the log-free bound~\eqref{pageeq} is neater and simpler to work with. 
\end{remark}

Now, with a view to proving Theorems \ref{upperthm} and \ref{lowerthm}, we fix $\gamma>0$, $B>\gamma^2+1$, let $D(\mcX)=\sqrt{\mcX}/(\log\mcX)^B$ and define
\begin{equation}\label{y0Q1def}
    y_0=y_0(\mcX):=\frac{\mcX}{(\log\mcX)^{B}}\quad\text{and}\quad Q_1:=Q_1(\mcX)=(\log y_0)^{B}.
\end{equation}
Next, we let $k_0=k_0(\mcX)$ denote the exceptional modulus up to $Q_1(\mcX)$ and
\begin{equation}
\label{eq:k1}
    k_1=k_1(\mcX):=
    \begin{cases}
        k_0,&\text{if $k_0$ exists, is square-free and $(k_0,\overline{\PP})=1$},\\
        0,&\text{otherwise.}
    \end{cases}
\end{equation}
Provided $k_1\neq 0$, we then let $q_1>\ldots>q_{\ell}$ be the prime divisors of $k_1$ and set
\begin{equation}\label{mjdef}
    m_j:=q_1\cdots q_j.
\end{equation}
Finally, for an infinite set of primes $\PP$ and multiplicative function $g(d)$, we set
\begin{equation}\label{jdefs}
    \mathcal{P}_j:=\mathcal{P}\setminus\{q_1,\ldots, q_j\}\quad\text{and}\quad V_j(z):=\prod_{\substack{p<z\\p\in\mathcal{P}_j}}\left(1-g(p)\right)
\end{equation}
with $V(z)=V_0(z)$ as in \eqref{vzdef}.

\subsection{Bounds relating to $\psi(x;d,a)$}

To obtain bounds relating to $\pi(x;d,a)$, we first work with the Chebyshev $\psi$ function, then apply partial summation. We let 
\begin{equation*}
    \Lambda(n):=
    \begin{cases}
        \log p,&\text{if $n=p^m$ for some prime $p$ and $m\geq 1$},\\
        0,&\text{otherwise.}
    \end{cases}
\end{equation*}
and
\begin{equation*}
    \psi(x;d,a):=\sum_{\substack{n\leq x\\ n\equiv a\ \text{(mod $d$)}}}\!\Lambda(n).
\end{equation*}
Our goal will be to obtain effective, averaged results for the error
\begin{equation}\label{psierr}
    \left|\psi(x;d,a)-\frac{\psi(x)}{\varphi(d)}\right|
\end{equation}
with $\psi(x)=\psi(x;1,0)$.

We begin with an averaged bound for the error \eqref{psierr} over moduli $Q_1< d\leq D$, which corrects a result of Sedunova~\cite{sedunova2019logarithmic}. As with Akbary and Hambrook's result (Theorem \ref{akhamthm}), this result is effective as it avoids the complications due to the Siegel zero at small moduli.

\begin{lemma}[{\cite[Corollary 1.4]{sedunova2019logarithmic}}]\label{sedunovalem}
    One has
    \begin{equation}\label{sedunovaeq}
        \sum_{Q_1< d\leq D}\sup_{y\leq \mcX}\max_{\substack{(a,d)=1}}\left|\psi(y;d,a)-\frac{\psi(y)}{\varphi(d)}\right|\ll_{B}\frac{\mcX}{(\log \mcX)^{B-3}},
    \end{equation}
    where the implied constant is effective.
\end{lemma}
\begin{proof}
    The proof of the lemma is essentially the same as that of~\cite[Corollary~1.4]{sedunova2019logarithmic}. However, we outline how to correct a small error in this result, which ultimately gives the denominator $(\log\mcX)^{B-3}$ in~\eqref{sedunovaeq} as opposed to $(\log\mcX)^{B-2}$.

    The key tool required is Vaughan's inequality, in which Sedunova~\cite[Proposition~1.3]{sedunova2019logarithmic} is able to prove in the form
    \begin{equation}\label{vaughaneq}
        \sum_{q\leq Q}\frac{q}{\varphi(q)}{\sum_{\chi(q)}}^*\sup_{y\leq x}|\psi(y,\chi)|\ll_{\varepsilon}(x+Q^2x^{1/2}+Qx^{13/14+\varepsilon})(\log x)^2
    \end{equation}
    for any $Q,\varepsilon>0$. Here, for a Dirichlet character $\chi$
    \begin{equation*}
        \psi(y,\chi):=\sum_{n\leq y}\Lambda(n)\chi(n)
    \end{equation*}
    and the starred sum in~\eqref{vaughaneq} indicates that one is only summing over primitive characters $\chi$ mod $q$. Arguing as in~\cite[\S 3]{sedunova2019logarithmic}, one then has
    \begin{align*}
        &\sum_{Q_1< d\leq D}\sup_{y\leq \mcX}\max_{\substack{(a,d)=1}}\left|\psi(y;d,a)-\frac{\psi(y)}{\varphi(d)}\right|\\
        &\qquad\qquad\qquad\qquad\ll\log \mcX\sum_{Q_1< d\leq D}\frac{1}{\varphi(d)}{\sum_{\chi(q)}}^*\sup_{2\leq y\leq \mcX}|\psi(y,\chi)|+D(\log\mcX)^3.
    \end{align*}
    Next, one uses partial summation to get that
    \begin{equation*}
        \sum_{Q_1< d\leq D}\frac{1}{\varphi(d)}{\sum_{\chi(q)}}^*\sup_{2\leq y\leq \mcX}|\psi(y,\chi)|=\frac{1}{D}\sum_{q\leq D}h(q)-\frac{1}{Q_1}\sum_{q\leq Q_1}h(q)+\int_{Q_1}^D\left(\sum_{q\leq t}h(q)\right)\frac{\mathrm{d}t}{t^2},
    \end{equation*}
    where
    \begin{equation*}
        h(q):=\frac{q}{\varphi(q)}{\sum_{\chi(q)}}^*\sup_{y\leq \mcX}|\psi(y,\chi)|.
    \end{equation*}
    Applying~\eqref{vaughaneq} thus gives
    \begin{equation*}
        \sum_{Q_1< d\leq D}\sup_{y\leq \mcX}\max_{\substack{(a,d)=1}}\left|\psi(y;d,a)-\frac{\psi(y)}{\varphi(d)}\right|\ll_\varepsilon\left(\frac{\mcX}{Q_1}+D\mcX^{1/2}+\mcX^{13/14+\varepsilon}\log\left(\frac{D}{Q_1}\right)\right)(\log \mcX)^3,
    \end{equation*}
    from which~\eqref{sedunovaeq} follows.
\end{proof}

To extend Lemma \ref{sedunovalem} to small moduli, we require the following pointwise bound which depends on the existence of a potential Siegel zero.

\begin{lemma}[See e.g.~{\cite[Lemma 6.10]{bennett2018explicit} or \cite[Theorem 1.2]{bordignon2021medium}}]\label{PNTAPlem}
    Let $d\leq Q_1$ and $\beta_d$ denote the Siegel zero modulo $d$ if it exists. Then,
    \begin{equation}\label{PNTAPeq}
        \sup_{y\leq\mcX}\max_{(a,d)=1}\left|\psi(y;d,a)-\frac{\psi(y)}{\varphi(d)}\right|\ll_B\frac{1}{\varphi(d)}\cdot
        \begin{cases}
            \frac{\mcX}{(\log \mcX)^{B}}+\mcX^{\beta_d},&\text{if}\ \beta_d\ \text{exists},\\
            \frac{\mcX}{(\log \mcX)^{B}},& \text{otherwise.}
        \end{cases}
    \end{equation}
    where the implied constant is effective and independent of $a$.
\end{lemma}

Throughout we will also make regular use of the following bound for the partial sums of $\mu^2(d)\gamma^{\omega(d)}/\varphi(d)$, which unsurprisingly appears when bounding a sieve remainder term of the form \eqref{RDeq}.

\begin{lemma}\label{muphilem}
    For all $\gamma>0$, we have
    \begin{equation*}
        \sum_{d\leq x}\frac{\mu^2(d)\gamma^{\omega(d)}}{\varphi(d)}\ll_{\gamma}(\log x)^{\gamma}
    \end{equation*}
    where the implied constant is effective.
\end{lemma}
\begin{proof}
    We have
    \begin{equation*}
       \sum_{d\leq x}\frac{\mu^2(d)\gamma^{\omega(d)}}{\varphi(d)}\leq\prod_{p\leq x}\left(1+\frac{\gamma}{p-1}\right)=\prod_{p\leq x}\left(1+\frac{\gamma}{p}\right)\prod_{p\leq x}\left(1+\frac{\gamma}{(p+\gamma)(p-1)}\right).
    \end{equation*}
    To finish, we note that
    \begin{equation*}
        \prod_{p\leq x}\left(1+\frac{\gamma}{(p+\gamma)(p-1)}\right)=O_{\gamma}(1)
    \end{equation*}
    and
    \begin{align}\label{logheq}
        \prod_{p\leq x}\left(1+\frac{\gamma}{p}\right)=\exp\left(\sum_{p\leq x}\log\left(1+\frac{\gamma}{p}\right)\right)&\ll\exp\left(\sum_{p\leq x}\frac{\gamma}{p}+O_{\gamma}(1)\right)\notag\\
        &\ll_\gamma(\log x)^{\gamma},
    \end{align}
    where in the last step we used Mertens' theorem.
\end{proof}

Equipped with these preliminary lemmas, we now prove a complete equidistribution result in the case where the exceptional modulus is small. Notably, if the exceptional zero is small enough then we can readily apply Lemma \ref{siegelboundlem} to bound the contribution of the Siegel zero in \eqref{PNTAPeq}.

\begin{proposition}\label{smallpsiprop}
    Suppose that $k_1$, defined in \eqref{eq:k1}, satisfies $k_1\leq\log\mcX$. Then,
    \begin{equation}\label{bomvinpsieq}
        \sum_{\substack{1\leq d\leq D\\(d,\overline{\PP})=1}}\mu^2(d)\sup_{y\leq \mcX}\max_{(a,d)=1}\left|\psi(y;d,a)-\frac{\psi(y)}{\varphi(d)}\right|\ll_B\frac{\mcX}{(\log\mcX)^{B-3}}
    \end{equation}
    where the implied constant is effective.
\end{proposition}
\begin{proof}
    We only need to consider the sum over $1\leq d\leq Q_1$ since the range ${Q_1< d\leq D}$ is covered by Lemma \ref{sedunovalem}.
    
    First, suppose that there does not exist a Siegel zero mod $d$ for all $d\leq Q_1$, then the desired result follows by applying Lemma \ref{PNTAPlem} pointwise for $d\leq Q_1$, followed by Lemma \ref{muphilem}. On the other hand, if there exists a Siegel zero $\beta_d$ induced by a character $\chi_d$ mod $d$, then there are three subcases to consider:
    \begin{enumerate}[label=(\alph*)]
        \item Suppose $k_1=0$. Recall the definition \eqref{eq:k1} of $k_1$. If $k_1=0$ then either $k_0$ does not exist, $k_0$ is not square-free or $(k_0,\overline{\mathcal{P}})>1$. Since our sum in \eqref{bomvinpsieq} only counts square-free $d$ with $(d,\overline{\mathcal{P}})=1$, it follows that $\beta_d$ is not induced by an exceptional character up to $Q_1$. As a result, Lemma \ref{exceptionalzerolem} gives
        \begin{equation*}
            1-\beta_d>\frac{R_3}{\log Q_1}\gg_B \frac{1}{\log\log \mcX}.
        \end{equation*}
        \item Suppose $k_1\neq 0$ and $k_1\nmid d$. Then, $\beta_d$ cannot be induced by the exceptional character mod $k_1$ and again we have
        \begin{equation*}
            1-\beta_d>\frac{R_3}{\log Q_1}\gg_B \frac{1}{\log\log \mcX}.
        \end{equation*}
        \item Finally suppose $k_1\mid d$. If $\chi_d$ is induced by the exceptional character mod $k_1$, Lemma \ref{siegelboundlem} gives
        \begin{equation*}
            1-\beta_d>\frac{R_2}{\sqrt{k_1}}\gg\frac{1}{\sqrt{\log \mcX}}
        \end{equation*}
        since $k_1\leq\log \mcX$. Otherwise, if $\chi_d$ is not induced by the exceptional character, we once again have
        \begin{equation*}
            1-\beta_d>\frac{R_3}{\log Q_1}\gg_B \frac{1}{\log\log \mcX}.
        \end{equation*}
    \end{enumerate}
    Combining the above three cases we obtain that 
    \begin{equation*}
        1-\beta_d\gg\frac{1}{\sqrt{\log \mcX}}.
    \end{equation*}
    Hence, by Lemmas \ref{PNTAPlem} and \ref{muphilem}, we have
    \begin{align*}
        \sum_{\substack{1\leq d\leq Q_1\\(d,\overline{\PP})=1}}\mu^2(d)\sup_{y\leq \mcX}\max_{(a,d)=1}\left|\psi(y;d,a)-\frac{\psi(y)}{\varphi(d)}\right|&\ll_B\sum_{\substack{1\leq d\leq Q_1\\(d,\overline{\mathcal{P}})=1}}\frac{\mu^2(d)}{\varphi(d)}\left(\frac{\mcX}{(\log \mcX)^{B}}+\mcX^{1-1/\sqrt{\log \mcX}}\right)\\
        &\ll_B\frac{\mcX\log\log\mcX}{(\log \mcX)^{B}}+\frac{\mcX\log\log\mcX}{\exp\left(\sqrt{\log \mcX}\right)},
    \end{align*}
    which is again sufficient.
\end{proof}

We also give a related result if we only sum over $k_1\nmid d$ to avoid any large contribution from a Siegel zero. This will be useful in the case where ${k_1>\log\mcX}$.

\begin{proposition}\label{largeequiprop}
    One has
    \begin{equation*}
        \sum_{\substack{1\leq d\leq D\\(d,\overline{\PP})=1\\k_1\nmid d}}\mu^2(d)\sup_{y\leq \mcX}\max_{(a,d)=1}\left|\psi(y;d,a)-\frac{\psi(y)}{\varphi(d)}\right|\ll_B\frac{\mcX}{(\log\mcX)^{B-3}}
    \end{equation*}
    where the implied constant is effective.
\end{proposition}
\begin{proof}
    Identical to the proof of Proposition \ref{smallpsiprop} except we do not include the case $k_1\mid d$ (i.e. Case (c)), which required $k_1\leq \log\mcX$.
\end{proof}

\subsection{Bounds relating to $\pi(x;d,a)$}

We now convert our results on $\psi(x;d,a)$ to results in terms of the usual prime-counting function $\pi(x;d,a)$. The proof of the following result is similar to that of \cite[Corollary 1.5]{sedunova2018partial} but we take more care so as to save an extra factor of $\log\mcX$. Moreover, to simplify our argument, we restrict to $y\in[y_0,\mcX]$ for now, but later weaken this restriction in our subsequent, more general result (Proposition \ref{mainPNTAPprop}).

\begin{proposition}\label{mainpiprop}
    We have
    \begin{equation}\label{piequieq}
        \sum_{\substack{1\leq d\leq D\\(d,\overline{\PP})=1\\k_1\nmid d}}\mu^2(d)\sup_{y_0\leq y\leq \mcX}\max_{(a,d)=1}\left|\pi(y;d,a)-\frac{\pi(y)}{\varphi(d)}\right|\ll_B\frac{\mcX}{(\log\mcX)^{B-2}}
    \end{equation}
    and if $k_1\leq\log \mcX$, then
    \begin{equation}\label{piequieq2}
        \sum_{\substack{1\leq d\leq D\\(d,\overline{\PP})=1}}\mu^2(d)\sup_{y_0\leq y\leq \mcX}\max_{(a,d)=1}\left|\pi(y;d,a)-\frac{\pi(y)}{\varphi(d)}\right|\ll_B\frac{\mcX}{(\log\mcX)^{B-2}}
    \end{equation}
    where the implied constant is effective.
\end{proposition}
\begin{proof}
    We only prove \eqref{piequieq2} since the proof of \eqref{piequieq} is essentially identical, using Proposition \ref{largeequiprop} in place of Proposition \ref{smallpsiprop}. 

    Let $y_0\leq y\leq\mcX$. We begin by defining
    \begin{equation*}
        \pi_1(y):=\sum_{2\leq n\leq y}\frac{\Lambda(n)}{\log n}\quad\text{and}\quad\pi_1(y;d,a):=\sum_{\substack{2\leq n\leq y\\n\equiv a\thinspace \text{(mod $d$)}}}\frac{\Lambda(n)}{\log n}.
    \end{equation*}
    Then, partial summation gives
    \begin{equation}\label{partsumpi1}
        \left|\pi_1(y;d,a)-\frac{\pi_1(y)}{\varphi(d)}\right|=\left|\frac{\psi(y;d,a)-\psi(y)/\varphi(d)}{\log y}-\int_2^y\frac{\psi(t;d,a)-\psi(t)/\varphi(d)}{t(\log t)^2}\mathrm{d}t\right|.
    \end{equation}
    Using \eqref{partsumpi1} and Proposition \ref{smallpsiprop}, we first show that \eqref{piequieq} holds with $\pi$ replaced with $\pi_1$. That is,
    \begin{equation}\label{pi1equieq}
        \sum_{\substack{1\leq d\leq D\\(d,\overline{\PP})=1}}\mu^2(d)\max_{(a,d)=1}\left|\pi_1(y;d,a)-\frac{\pi_1(y)}{\varphi(d)}\right|\ll_B\frac{\mcX}{(\log\mcX)^{B-2}}
    \end{equation}
    uniformly for $y\in[y_0,\mcX]$.
    We split the integral in \eqref{partsumpi1} into the ranges $2\leq t\leq y^{0.1}$ and $y^{0.1}\leq t\leq y$. In the first case, the trivial bounds ${\psi(t;d,a)\leq \psi(t)\leq t\log t}$ give
    \begin{equation}\label{smallinteq}
        \int_2^{y^{0.1}}\frac{\psi(t;d,a)-\psi(t)/\varphi(d)}{t(\log t)^2}\mathrm{d}t\ll \int_{2}^{y^{0.1}}\frac{1}{\log t}\mathrm{d}t\ll \mcX^{0.1}.
    \end{equation}
    Applying the triangle inequality to \eqref{partsumpi1} and substituting in \eqref{smallinteq} then yields
    \begin{align}
        &\left|\pi_1(y;d,a)-\frac{\pi_1(y)}{\varphi(d)}\right|\notag\\\
        &\quad\ll\frac{1}{\log y}\left|\psi(y;d,a)-\frac{\psi(y)}{\varphi(d)}\right|+\sup_{y^{0.1}\leq t\leq y}\left|\psi(t;d,a)-\frac{\psi(t)}{\varphi(d)}\right|\int_{y^{0.1}}^y\frac{1}{t(\log t)^2}\mathrm{d}t+\mcX^{0.1}\notag\\
        &\quad\ll\frac{1}{\log\mcX}\sup_{y^{0.1}\leq t\leq y}\left|\psi(t;d,a)-\frac{\psi(t)}{\varphi(d)}\right|+\mcX^{0.1}.\notag
    \end{align}
    Hence, using Proposition \ref{smallpsiprop} gives \eqref{pi1equieq} as desired. To finish, we note that
    \begin{equation}\label{pi1eq1}
        \pi_1(y;d,a)-\pi(y;d,a)=\sum_{2\leq k\leq\frac{\log y}{\log 2}}\sum_{\substack{p^k\leq y\\p^k\equiv a\thinspace \text{(mod $d$)}}}\frac{1}{k}\leq\sum_{2\leq k\leq\frac{\log y}{\log 2}}\frac{\pi(\sqrt{y})}{2}\ll\sqrt{y},    
    \end{equation}
    where the last inequality follows from the prime number theorem. Similarly, 
    \begin{equation}\label{pi1eq2}
        \pi_1(y)-\pi(y)\ll \sqrt{y}.
    \end{equation}
    To finish, we use \eqref{pi1equieq}, \eqref{pi1eq1} and \eqref{pi1eq2} to give
    \begin{equation*}
        \sum_{\substack{1\leq d\leq D\\(d,\overline{\PP})=1}}\mu^2(d)\max_{(a,d)=1}\left|\pi(y;d,a)-\frac{\pi(y)}{\varphi(d)}\right|\ll_B\frac{\mcX}{(\log\mcX)^{B-2}}+\sqrt{y}D\ll\frac{\mcX}{(\log\mcX)^{B-2}}
    \end{equation*}
    as required.
\end{proof}

For the purposes of bounding a sieve remainder term of the form \eqref{RDeq}, we now wish to incorporate a factor of $\gamma^{\omega(d)}$ into the above result. For this, we first recall the following bound on the moments of the divisor function $\tau(d)$.
\begin{lemma}[{See \cite[Theorem~318]{hardy2008introduction} and \cite{luca2017rth}}]\label{divmomlem}
    Let $\tau(d)$ denote the number of divisors of $d$. Then, for any integer $r\geq 1$, one has
    \begin{equation}\label{momupeq}
        \sum_{d\leq x}\tau(d)^r\ll_r x(\log x)^{2^r-1}.
    \end{equation}
\end{lemma}
\begin{remark}
    In~\cite{luca2017rth} much more is shown than Lemma~\ref{divmomlem}. In particular, an asymptotic for $\sum_{d\leq x}\tau(d)^r$ is established via elementary means. However, the upper bound~\eqref{momupeq} is more than sufficient for our purposes.
\end{remark}

Our main Bombieri--Vinogradov style result is then as follows.

\begin{proposition}\label{mainPNTAPprop}
    With $B_{\gamma}$ as in \eqref{Bgamdef}, one has
    \begin{equation}\label{pilargegammaeq}
        \sum_{\substack{1\leq d\leq D\\(d,\overline{\PP})=1\\k_1\nmid d}}\mu^2(d)\gamma^{\omega(d)}\sup_{y\leq \mcX}\max_{(a,d)=1}\left|\pi(y;d,a)-\frac{\pi(y)}{\varphi(d)}\right|\ll_{B,\gamma}\frac{\mcX}{(\log\mcX)^{B_{\gamma}}}
    \end{equation}
    and if $k_1\leq\log \mcX$,
    \begin{equation}\label{pismallgammaeq}
        \sum_{\substack{1\leq d\leq D\\(d,\overline{\PP})=1}}\mu^2(d)\gamma^{\omega(d)}\sup_{y\leq \mcX}\max_{(a,d)=1}\left|\pi(y;d,a)-\frac{\pi(y)}{\varphi(d)}\right|\ll_{B,\gamma}\frac{\mcX}{(\log\mcX)^{B_{\gamma}}}
    \end{equation}
    where the implied constants are effective.
\end{proposition}
\begin{proof}
    We only prove \eqref{pismallgammaeq} as the proof of \eqref{pilargegammaeq} is essentially identical. The case $\gamma\leq 1$ follows immediately from Proposition \ref{smallpsiprop} so we assume $\gamma>1$ throughout. We split into three cases for $y$.\\
    \\
    \textbf{Case 1:} $2\leq y\leq \sqrt{\mcX}$.\\
    In this case, we use the trivial bounds $\pi(y;d,a)\leq \pi(y)\leq y$ to obtain
    \begin{align}\label{maincase1}
        \sum_{\substack{1\leq d\leq D\\(d,\overline{\PP})=1}}\mu^2(d)\gamma^{\omega(d)}\max_{(a,d)=1}\left|\pi(y;d,a)-\frac{\pi(y)}{\varphi(d)}\right|&\ll y\sum_{1\leq d\leq D}\mu^2(d)\gamma^{\omega(d)}\notag\\
        &\ll\sqrt{\mcX}\sum_{1\leq d\leq D}\mu^2(d)\gamma^{\omega(d)}.
    \end{align}
    We now note that for a square-free number $d$, one has $\tau(d)=2^{\omega(d)}$. Hence,
    \begin{equation}\label{divisortrick}
        \sum_{1\leq d\leq D}\mu^2(d)\gamma^{\omega(d)}=\sum_{1\leq d\leq D}\mu^2(d)2^{\omega(d)(\log\gamma/\log 2)}\ll\sum_{1\leq d\leq D}\tau(d)^{\lceil\log\gamma/\log 2\rceil}.
    \end{equation}
    Lemma \ref{divmomlem} then gives
    \begin{equation*}
        \sum_{1\leq d\leq D}\tau(d)^{\lceil\log\gamma/\log 2\rceil}\ll_\gamma D(\log D)^{2\gamma-1}\ll_{B,\gamma}\frac{\sqrt{\mcX}}{(\log\mcX)^{B-2\gamma+1}}
    \end{equation*}
    so that \eqref{maincase1} is bounded by
    \begin{equation*}
        \sum_{\substack{1\leq d\leq D\\(d,\overline{\PP})=1}}\mu^2(d)\gamma^{\omega(d)}\max_{(a,d)=1}\left|\pi(y;d,a)-\frac{\pi(y)}{\varphi(d)}\right|\ll_{B,\gamma}\frac{\mcX}{(\log\mcX)^{B-2\gamma+1}}.
    \end{equation*}
    This bound is sufficient since 
    \begin{equation*}
        B-2\gamma+1\geq B_{\gamma}:=(B-1-\gamma^2)/2
    \end{equation*} for $B>\gamma^2+1$.\\
    \\
    \textbf{Case 2:} $\sqrt{\mcX}\leq y\leq y_0$.\\
    In this case, the Brun--Titchmarsh theorem (see e.g.~\cite[Theorem~2]{montgomery1973large}) gives
    \begin{equation*}
        \pi(y;d,a)\ll_B\frac{y}{\varphi(d)\log\log\mcX}\quad\text{and}\quad\frac{\pi(y)}{\varphi(d)}\ll\frac{y}{\varphi(d)\log y}
    \end{equation*}
    for all $d\leq D$. Using these bounds and Lemma \ref{muphilem},
    \begin{align*}
        \sum_{\substack{1\leq d\leq D\\(d,\overline{\PP})=1}}\mu^2(d)\gamma^{\omega(d)}\max_{(a,d)=1}\left|\pi(y;d,a)-\frac{\pi(y)}{\varphi(d)}\right|&\ll_B\frac{y_0}{\log\log\mcX}\sum_{d\leq D}\frac{\mu^2(d)\gamma^{\omega(d)}}{\varphi(d)}\\
        &\ll_{B,\gamma}\frac{\mcX}{(\log\mcX)^{B-\gamma}\log\log\mcX},
    \end{align*}
    which is again sufficient.\\
    \\
    \textbf{Case 3:} $y_0\leq y\leq \mcX$.\\
    In this case, we begin by writing
    \begin{equation*}
        E(d)=\mu^2(d)\max_{(a,d)=1}\left|\pi(y;d,a)-\frac{\pi(y)}{\varphi(d)}\right|\quad\text{and}\quad f(d)=\mu^2(d)\gamma^{\omega(d)}.
    \end{equation*}
    By the Cauchy--Schwarz inequality,
    \begin{align}\label{cseq}
        \sum_{\substack{d\leq D\\(d,\overline{\PP})=1}}f(d)E(d)&\leq\sqrt{{\sum_{\substack{d\leq D\\(d,\overline{\PP})=1}}f(d)^2E(d)}}\sqrt{\sum_{\substack{d\leq D\\(d,\overline{\PP})=1}}E(d)}.
    \end{align}
    Here, Proposition \ref{mainpiprop} gives
    \begin{equation}\label{cseq1}
        \sum_{\substack{d\leq D\\(d,\overline{\PP})=1}}E(d)\ll_B\frac{\mcX}{(\log\mcX)^{B-2}}.
    \end{equation}
    Next, an application of the Brun--Titchmarsh theorem and Lemma \ref{muphilem} yields
    \begin{equation}\label{cseq2}
        \sum_{\substack{d\leq D\\(d,\overline{\PP})=1}}f(d)^2E(d)\ll\frac{\mcX}{\log\mcX}\sum_{\substack{d\leq D\\(d,\overline{\PP})=1}}\frac{f(d)^2}{\varphi(d)}\ll_{\gamma}\frac{\mcX(\log\mcX)^{\gamma^2}}{\log\mcX}.
    \end{equation}
    The desired result then follows from substituting \eqref{cseq1} and \eqref{cseq2} into \eqref{cseq}.
\end{proof}

\begin{remark}
    As one might expect, a small improvement is possible if one uses H\"older's inequality in place of the Cauchy--Schwarz inequality. For example, provided ${B>\gamma+3}$, one can apply H\"older's inequality in the proof above, writing
    \begin{equation*}
        \sum_{\substack{d\leq D\\(d,\overline{\PP})=1}}f(d)E(d)\leq\Bigg({\sum_{\substack{d\leq D\\(d,\overline{\PP})=1}}f(d)^rE(d)}\Bigg)^{1/r}\Bigg(\sum_{\substack{d\leq D\\(d,\overline{\PP})=1}}E(d)\Bigg)^{1/s}.
    \end{equation*}
    in place of \eqref{cseq} with
    \begin{equation*}
        (r,s)=\left(\frac{\log(B-3)}{\log\gamma},\frac{\log(B-3)}{\log(B-3)-\log\gamma}\right).
    \end{equation*}
    This yields a bound in \eqref{pilargegammaeq} and \eqref{pismallgammaeq} of $\mcX/(\log\mcX)^{\widetilde{B}_\gamma}$, with
    \begin{equation*}
        \widetilde{B}_\gamma=B-2-\frac{2(B-2)\log\gamma}{\log(B-3)}.  
    \end{equation*}
    So, for instance, if $B=19$ and $\gamma=2$ then $\widetilde{B}_{\gamma}=8.5$ which is greater than ${B_{\gamma}=7}$. Nevertheless, we use $B_{\gamma}$ in the statement of Theorems \ref{upperthm} and \ref{lowerthm} for simplicity.
\end{remark}

As we shall see in the subsequent section, Proposition \ref{mainPNTAPprop} is sufficient to prove Theorem \ref{upperthm}. However, to prove the effective sieve lower bound, Theorem \ref{lowerthm}, we also require equidistribution results for the error
\begin{equation}\label{kderr}
    \left|\pi(y;k d,a)-\frac{\pi(y;k,a)}{\varphi(d)}\right|.
\end{equation}
Essentially, we show (Proposition \ref{pimidkprop}) that one can obtain an effective Bombieri--Vinogradov type theorem for \eqref{kderr} whenever $k\mid k_1$ with $k\neq 1$.

To this aim, we work with the character version of Chebyshev's $\psi$ function:
\begin{equation*}
    \psi(x,\chi)=\sum_{n\leq x}\Lambda(n)\chi(n),
\end{equation*}
where $\chi$ is a Dirichlet character. By orthogonality, one has the relation
\begin{equation}\label{psichirel}
    \psi(x;d,a) = \frac{1}{\varphi(d)} \sum_{\chi(d)} \overline{\chi}(a) \psi(x, \chi).
\end{equation}
This leads to the following variant of Lemma \ref{PNTAPlem}.

\begin{lemma}\label{PNTAPlem3}
    Let $d\leq Q_1$ and $\beta_d$ denote the Siegel zero modulo $d$ if it exists. Then,
    \begin{equation}\label{PNTAP3eq}
        \sup_{y\leq\mcX} \max_{(a,d)=1}\left|\sum_{\substack{\chi\ (\mathrm{mod}\ d)\\\chi\neq\chi_0}}\overline{\chi}(a)\psi(y,\chi)\right|\ll_B
        \begin{cases}
            \frac{\mcX}{(\log \mcX)^{B}}+\mcX^{\beta_d},&\text{if}\ \beta_d\ \text{exists},\\
            \frac{\mcX}{(\log \mcX)^{B}},& \text{otherwise,}
        \end{cases}
    \end{equation}
    where the implied constant is effective.
\end{lemma}
\begin{proof}
    Let $y\leq\mcX$ and $(a,d)=1$. From \eqref{psichirel}, we have
    \begin{equation*}
        \left|\sum_{\substack{\chi\ (\mathrm{mod}\ d)\\\chi\neq\chi_0}}\overline{\chi}(a)\psi(y,\chi)\right|=\varphi(d)\left|\psi(y;d,a)-\frac{\psi(y,\chi_0)}{\varphi(d)}\right|,
    \end{equation*}
    where $\chi_0$ is the principal character mod $d$. Now,
    \begin{align}\label{psiprimbound}
        \left|\psi(y)-\psi(y,\chi_0)\right|\le
        \sum_{\substack{p^m\le y\\p^m\mid d}}\log p\leq\log y\sum_{\substack{p^m\le y\\p^m\mid d}}1\leq \log y\sum_{m\leq \frac{\log y}{\log 2}}\sum_{p\mid d}1\ll(\log y)^3
    \end{align}
    so that the desired result follows from Lemma \ref{PNTAPlem}.
\end{proof}

Our final equidistribution result is then as follows.

\begin{proposition}\label{pimidkprop}
    With $B_\gamma$ as in \eqref{Bgamdef}, we have for each $k\mid k_1$ with $k\neq 1$,
    \begin{equation}\label{pimidkeq}
        {\sum_{\substack{d\leq D/k\\(d,\overline{\PP})=1\\(d,k)=1}}}^{\!\!\!\!\!\sharp}\ \mu^2(d)\gamma^{\omega(d)}\sup_{y\leq \mcX}\max_{(a,k d)=1}\left|\pi(y;k d,a)-\frac{\pi(y;k,a)}{\varphi(d)}\right|\ll_{B,\gamma}\frac{\mcX}{(\log\mcX)^{B_{\gamma}}},
    \end{equation}
    where $\sharp$ means that the sum is restricted to $d$ with $k_1\nmid kd$ if $k\neq k_1$.
\end{proposition}
\begin{proof}
    The proof is similar to \cite[Lemma~4.10]{BJV2025} but we give the full details for completeness. We also remark that in~\cite[Lemma~4.10]{BJV2025} and its proof there are a couple of small errors, particularly the lack of the $\sharp$ restriction on the moduli $d$. Thus, we offer a corrected statement and proof here.
    
    First we prove the corresponding result for the Chebyshev $\psi$ function and $\gamma=1$. That is,
    \begin{equation}\label{psimidkeq}
        {\sum_{\substack{d\leq D/k\\(d,\overline{\PP})=1\\(d,k)=1}}}^{\!\!\!\!\!\sharp}\ \mu^2(d)\sup_{y\leq \mcX}\max_{(a,k d)=1}\left|\psi(y;k d,a)-\frac{\psi(y;k,a)}{\varphi(d)}\right|\ll_{B}\frac{\mcX}{(\log\mcX)^{B-3}}.
    \end{equation}
    Let $y\leq\mcX$ and $d,k,a\in\mathbb{N}$ such that $(d,\overline{\mathcal{P}})=1$, $(d,k)=1$ and $(a,kd)=1$. Since $(k,d)=1$ any character modulo $kd$ can be represented in a unique way as the product of two characters modulo $k$ and modulo $d$. Therefore,
    \begin{align}
        &\left|\psi(y;k d,a)-\frac{\psi(y;k,a)}{\varphi(d)}\right|\notag\\
        &=\frac{1}{\varphi(kd)}\sum_{\substack{\chi\:(\text{mod } kd)}}\overline{\chi}(a)\psi(y,\chi) - \frac{1}{\varphi(d)\varphi(k)}\sum_{\substack{\chi_1\:(\text{mod } k)}}\overline{\chi_1}(a)\psi(y,\chi_1)\notag \\
        &=\frac{1}{\varphi(kd)}\left[\sum_{\substack{\chi_1\:(\text{mod } k) \\ \chi_2\: (\text{mod }d)}}\overline{\chi_1}(a)\overline{\chi_2}(a)\psi(y,\chi_1\chi_2) - \sum_{\substack{\chi_1\:(\text{mod } k)}}\overline{\chi_1}(a)\overline{\chi_{0,d}}(a)\psi(y,\chi_1)\right] \notag\\
        &=\frac{1}{\varphi(kd)}
        \left[\sum_{\substack{\chi_1\:(\text{mod } k) \\ \chi_2\: (\text{mod }d)}}\overline{\chi_1}(a)\overline{\chi_2}(a)\psi(y,\chi_1\chi_2) - \sum_{\substack{\chi_1\:(\text{mod } k)}}\overline{\chi_1}(a)\overline{\chi_{0,d}}(a)\psi(y,\chi_1 \chi_{0,d})\right]\notag\\
        &\qquad\qquad\qquad\qquad-\frac{1}{\varphi(kd)}\sum_{\substack{\chi_1\:(\text{mod } k)}}\overline{\chi_1}(a)\overline{\chi_{0,d}}(a)\left(\psi(y,\chi_1)-\psi(y,\chi_1 \chi_{0,d})\right) \notag\\
        &=\frac{1}{\varphi(kd)}
        \sum_{\substack{\chi_1\:(\text{mod } k) \\ \chi_2\ne \chi_{0,d}\: (\text{mod }d)}}\overline{\chi_1}(a)\overline{\chi_2}(a)\psi(y,\chi_1\chi_2) \notag\\
        &\qquad\qquad\qquad\qquad-\frac{1}{\varphi(kd)}\sum_{\substack{\chi_1\:(\text{mod } k)}}\overline{\chi_1}(a)\overline{\chi_{0,d}}(a)\left(\psi(y,\chi_1)-\psi(y,\chi_1 \chi_{0,d})\right),\label{psikdeq}
    \end{align}
    where $\chi_{0,d}$ is the principal character modulo $d$. 

    We deal with the second term in \eqref{psikdeq} first. Here, one has (cf.\ \eqref{psiprimbound})
    \begin{equation*}
        \left|\psi(y,\chi_1)-\psi(y,\chi_1 \chi_{0,d})\right|\ll (\log y)^3.
    \end{equation*}
    Hence, by Lemma \ref{muphilem},
    \begin{align}
         {\sum_{\substack{d\leq D/k\\(d,\overline{\PP})=1\\(d,k)=1}}}^{\!\!\!\!\!\sharp}\ \frac{\mu^2(d)}{\varphi(kd)}\left|\sum_{\substack{\chi_1\:(\text{mod } k)}}\overline{\chi_1}(a)\overline{\chi_{0,d}}(a)\left(\psi(y,\chi_1)-\psi(y,\chi_1 \chi_{0,d})\right)\right|&\ll \sum_{d\leq D}\frac{\mu^2(d)}{\varphi(d)}(\log y)^3\notag\\
         &\ll(\log\mcX)^{4}.\label{eq: 534error}
    \end{align}
    Thus, to obtain \eqref{psimidkeq} the only thing left to show is that
    \begin{align}
        {\sum_{\substack{d\leq D/k\\(d,\overline{\PP})=1\\(d,k)=1}}}^{\!\!\!\!\!\sharp}\ \frac{\mu^2(d)}{\varphi(kd)}
        \left| \sum_{\substack{\chi_1\:(\text{mod } k) \\ \chi_2\ne \chi_{0,d}\: (\text{mod }d)}}  \overline{\chi_1}(a)\overline{\chi_2}(a)\psi(y,\chi_1\chi_2)\right|\ll_B\frac{\mcX}{(\log\mcX)^{B-3}}.\label{psichareq}
    \end{align}
    However, this can be proven in an identical fashion to Proposition \ref{smallpsiprop}, using Lemma \ref{PNTAPlem3} as opposed to Lemma \ref{PNTAPlem}. Here we also do not need the restriction $k_1\leq\log\mcX$ since the conditions $\chi_2\neq\chi_{0,d}$, $k\mid k_1$ and the $\sharp$ restriction on the moduli $d$ ensure that the exceptional character never appears in the inner sum of \eqref{psichareq}.

    To convert this to a result in terms of the prime counting function $\pi$ and arbitrary ${\gamma>0}$, one uses partial summation as in the proof of Proposition \ref{mainpiprop}, followed by the Cauchy--Schwarz inequality as in the proof of Proposition \ref{mainPNTAPprop}, giving \eqref{pimidkeq}.
\end{proof}

\section{Proof of the effective upper bound}\label{uppersect}
We now prove Theorem \ref{upperthm}. The proof is short given Proposition \ref{mainPNTAPprop}. All we require is the following upper bound for $V_1(z)$ when $k_1>\log\mcX$. This will correspond to the $O_A(1/\log\log X)$ term in \eqref{sieveubnew}.
\begin{lemma}\label{V1lem}
    Let $k_1$ be as in \eqref{eq:k1} and $V_j(z)$ be as in \eqref{jdefs} for some set of primes $\PP$ and multiplicative function $g(d)$. Suppose that the conditions \eqref{paconds} hold. Then, if $k_1>\log\mcX$,
    \begin{equation}\label{V1toVeq}
        V_1(z)\leq\left(1+O_{A}\left(\frac{1}{\log\log \mcX}\right)\right)V(z).
    \end{equation}
    \begin{proof}
       By the definition~\eqref{jdefs} of $V_1(z)$,
        \begin{equation*}
            V_1(z)\leq\frac{V(z)}{1-g(q_1)}.
        \end{equation*}
        Applying the condition $g(q_1)\leq A/(q_1-1)$ (see~\eqref{paconds}) then gives
        \begin{equation*}
            V_1(z)\leq\frac{V(z)}{1-\frac{A}{q_1-1}}=\left(1+\frac{A}{q_1-1-A}\right)V(z).
        \end{equation*}
        Thus, if we can show that $q_1\gg\log\log \mcX$ then we are done. In this direction, first recall that $k_1$ is square-free by its definition~\eqref{eq:k1}. Thus,
        \begin{equation*}
            \prod_{p\leq q_1}p\geq k_1> \log\mcX.
        \end{equation*}
        Taking logarithms,
        \begin{equation*}
            \theta(q_1)>\log\log\mcX,
        \end{equation*}
        where $\theta(x)=\sum_{p\leq x}\log p$ is the Chebyshev theta function. The desired bound $q_1\gg\log\log \mcX$ then follows by the prime number theorem.
    \end{proof}
\end{lemma}

We now complete the proof of Theorem \ref{upperthm}. 
\begin{proof}
    We consider two cases: $k_1\leq\log\mcX$ and $k_1>\log\mcX$.\\
    \\
    \textbf{Case 1:} $k_1\leq\log\mcX$.
    Here, the result follows from \eqref{pismallgammaeq} of Proposition \ref{mainPNTAPprop}.\\
    \\
    \textbf{Case 2:} $k_1>\log\mcX$.
    In this case, we first note that
    \begin{equation}\label{SA1upper}
        S(\mcA,\PP,z)\leq S(\mcA,\PP_1,z)<XV_1(z)\left(F(s)+\varepsilon(X)\right)+\sum_{\substack{d<D\\(d,\overline{\mathcal{P}})=1\\ q_1\nmid d}}\mu^2(d)r(d)
    \end{equation}
    with $\mathcal{P}_1$ as defined in \eqref{jdefs}. The result now follows by bounding $V_1(z)$ via Lemma \ref{V1lem} and bounding the remainder term in \eqref{SA1upper} via \eqref{pilargegammaeq} of Proposition \ref{mainPNTAPprop}.
\end{proof}

\section{Proof of the effective lower bound}\label{lowersect}
In this section we prove the effective sieve lower bound, Theorem \ref{lowerthm}. The proof is more delicate than the proof of Theorem \ref{upperthm}, and so we begin by proving a series of preliminary lemmas. 

\subsection{Preliminary lemmas}
In the following lemmas, we use the notation and conditions in Theorem \ref{lowerthm}, as well as the notation from Section \ref{setupnotsub}.

To begin with, we use an inclusion-exclusion argument to relate $S(\mcA,\PP,z)$ to sifting functions in terms of $\mcA_{m_j}$ and $\PP_j$. This is the key identity that will allow us to circumvent the Siegel zero in the sieving error term. This lemma is essentially the same as \cite[Lemma 6.3]{BJV2025}, but we include the short proof here for completeness.

\begin{lemma}\label{exinclem}
    We have
    \begin{equation}\label{inclexcleq}
        S(\mcA,\PP,z)=\sum_{j=0}^{\ell-1}(-1)^j S(\mcA_{m_j},\PP_{j+1},z)+(-1)^{\ell} S(\mcA_{m_{\ell}},\PP_\ell,z),
    \end{equation}
    with $m_j$ and $\PP_j$ as defined in~\eqref{mjdef} and~\eqref{jdefs} respectively.
\end{lemma}
\begin{proof}
    First, see that $S(\mcA,\PP_1,z)-S(\mcA,\PP,z)$ counts the number of integers in $\mcA$ that are divisible by $q_1$ but not by any other primes in $\PP$ less than $z$. Then, ${S(\mcA_{m_1},\PP_2,z)-S(\mcA,\PP_1,z)+S(\mcA,\PP,z)}$ counts the integers in $\mcA$ that are divisible by $q_1$ and $q_2$ but not by any other primes in $\PP$ less than $z$. More generally $\sum_{j=0}^{\ell-1}(-1)^{\ell-1-j}S(\mcA_{m_j},\PP_{j+1},z)+(-1)^{\ell}S(\mcA,\PP,z)$ counts the integers in $\mcA$ divisible by $q_1,\ldots,q_{\ell}$ but by no other primes in $\PP$ less than $z$. That is,
    \begin{equation*}
        S(\mcA_{m_\ell},\PP_\ell,z)=\sum_{j=0}^{\ell-1}(-1)^{\ell-1-j}S(\mcA_{m_j},\PP_{j+1},z)+(-1)^{\ell}S(\mcA,\PP,z),
    \end{equation*}
    which rearranges to give the desired result.
\end{proof}
We also have the following two related results, which concern the functions $V_j(z)$. Here, noting that $s>1$ in Theorem \ref{lowerthm}, we will make use of the fact that
\begin{equation}\label{zseq}
    \log z\asymp_{B,s}\log\mcX.
\end{equation}
\begin{lemma}\label{Vinductlem}
    Provided $k_1\neq 0$, one has
    \begin{equation}\label{Vinducteq}
        V(z)=\sum_{j=0}^{\ell-1}(-1)^j g(m_j)V_{j+1}(z)+(-1)^{\ell}g(m_{\ell})V_{\ell}(z).
    \end{equation}
\end{lemma}
\begin{proof}
    Taking $\mcX$ to be sufficiently large, by \eqref{zseq} we may assume $z>k_1$ throughout. In particular, we have
    \begin{equation}\label{Vj1def}
        V_{j+1}(z)=\prod_{\substack{p<z\\p\in\PP\\p\neq q_1,\ldots,q_{j+1}}}(1-g(p))=\frac{V(z)}{\prod_{i=1}^{j+1}(1-g(q_i))}.
    \end{equation}
    We proceed by induction on $\ell$. For the base case $\ell=1$, by \eqref{Vj1def} the left-hand side of \eqref{Vinducteq} is
    \begin{equation*}
        V_1(z)-g(q_1)V_1(z)=\frac{(1-g(q_1))V(z)}{(1-g(q_1))}=V(z),
    \end{equation*}
    as required. Now, suppose that \eqref{Vinducteq} holds for $\ell=k$. Then, for $\ell=k+1$ we first note that
    \begin{align*}
        (-1)^kg(m_k)V_{k+1}(z)+(-1)^{k+1}g(m_{k+1})V_{k+1}(z)&=(-1)^kg(m_k)(1-g(q_{k+1}))V_{k+1}(z)\\
        &=(-1)^kg(m_k)V_k(z)
    \end{align*}
    by \eqref{Vj1def}. Hence, by the induction hypothesis
    \begin{align*}
        &\sum_{j=0}^{k}(-1)^j g(m_j)V_{j+1}(z)+(-1)^{k+1}g(m_{k+1})V_{k+1}(z)\\
        &\qquad\qquad\qquad\qquad\qquad\qquad\qquad=\sum_{j=0}^{k-1}(-1)^j g(m_j)V_{j+1}(z)+(-1)^{k}g(m_{k})V_{k}(z)\\
        &\qquad\qquad\qquad\qquad\qquad\qquad\qquad=V(z).
    \end{align*}
    This completes the proof.
\end{proof}

\begin{lemma}\label{v1vlem}
    If $k_1\neq 0$, one has
    \begin{equation}\label{v1veq}
        \sum_{j=1}^{\ell-1}g(m_j)V_{j+1}(z)+g(m_{\ell})V_\ell(z)\ll_A V(z).
    \end{equation}
\end{lemma}
\begin{proof}
    We begin by extending our notation for $q_i$ (defined in Section~\ref{setupnotsub}) by setting
    \begin{equation*}
        q_i=0,\qquad i>\ell    
    \end{equation*}
    and use the convention that $g(0)=0$. Hence, with this notation
    \begin{equation*}
        V_{\ell+1}(z)=\frac{V_{\ell}(z)}{1-g(q_{\ell+1})}=V_{\ell}(z).
    \end{equation*}
    By further noting that (see Lemma~\ref{V1lem}) 
    \begin{equation*}
        V_1(z)-V(z)\ll_A V(z),
    \end{equation*}
    we see that it is sufficient to prove
    \begin{equation}\label{v1veq2}
        \sum_{j=1}^{\ell}g(m_j)V_{j+1}(z)\ll_{A}V_1(z)-V(z)
    \end{equation}
    in place of~\eqref{v1veq}.  Next, we note that
    \begin{equation*}
        V_1(z)-V(z)=\frac{V(z)}{1-g(q_1)}-V(z)=\frac{g(q_1)V(z)}{1-g(q_1)}
    \end{equation*}
    and thus,
    \begin{align}
        g(m_j)V_{j+1}(z)&=\prod_{i=1}^{j}{g(q_i)}\frac{V(z)}{\prod_{i=1}^{j+1}(1-g(q_i))}\notag\\
        &=\frac{V_1(z)-V(z)}{1-g(q_{j+1})}\prod_{i=2}^j\frac{g(q_i)}{1-g(q_i)}.\label{gmjvjeq}
    \end{align}
    In order to bound~\eqref{gmjvjeq}, we split into two cases $1\leq j<j'$ and $j\geq j'$, where $j'$ is defined to be the largest integer with
    \begin{equation}\label{jpdef}
        q_{i}\geq A+2\quad\text{for all}\ i\in\{1,\ldots, j'\},
    \end{equation}
    meaning also that there are only $O_A(1)$ values of $j$ with $j'\leq j\leq\ell$.\\
    \\
    \textbf{Case 1:} $1\leq j< j'$.\\
    In this case, we substitute the bound $g(q_i)\leq A/(q_i-1)$ from \eqref{paconds} and the bound~\eqref{jpdef} into~\eqref{gmjvjeq} to yield
    \begin{align}
        g(m_j)V_{j+1}(z)&\leq\frac{V_1(z)-V(z)}{1-A/(q_{j+1}-1)}\frac{A^{j-1}}{\prod_{i=2}^j(q_i-1-A)}\notag\\
        &\ll_A(V_1(z)-V(z))\frac{A^{j-1}}{\prod_{i=2}^j(q_i-1-A)}.\label{gmjvjeq2}
    \end{align}
    Since each $q_i$ in the product in~\eqref{gmjvjeq2} is distinct and satisfies $q_i\geq A+2$ by~\eqref{jpdef}, \eqref{gmjvjeq2} can be bounded further as
    \begin{equation}\label{vj1final}
        g(m_j)V_{j+1}(z)\ll_A (V_1(z)-V(z))\frac{A^{j-1}}{(j-1)!}.
    \end{equation}
    \ \\
    \textbf{Case 2:} $j\geq j'$.\\
    In this case, we first note that since there are only $O_{A}(1)$ values of $j$ with $j'<j\leq\ell$,
    \begin{equation*}
       \prod_{j>j'} \frac{1}{1-g(q_j)}=O_{A}(1).
    \end{equation*}
    Therefore, from~\eqref{gmjvjeq} and the fact that $g(q_i)<1$, we have for $j\geq j'$,
    \begin{align*}
        g(m_j)V_{j+1}(z)&\ll_A(V_1(z)-V(z))\prod_{i=2}^{j'}\frac{g(q_i)}{1-g(q_i)}\\
        &\ll(V_1(z)-V(z))\frac{A^{j'-1}}{(j'-1)!}
    \end{align*}
    with the second bound following from the argument in Case 1.\\
    \\
    Combining Cases 1 and 2,
    \begin{align*}
        \sum_{j=1}^{\ell}g(m_j)V_{j+1}(z)\ll_{A}(V_1(z)-V(z))\left(\sum_{1\leq j<j'}\frac{A^{j-1}}{(j-1)!}+\sum_{j'\leq j\leq\ell}\frac{A^{j'-1}}{(j'-1)!}\right).
    \end{align*}
    Here,
    \begin{equation*}
        \sum_{1\leq j<j'}\frac{A^{j-1}}{(j-1)!}\leq e^A
    \end{equation*}
    and since $(j-1)!$ grows faster than $A^{j-1}$,
    \begin{equation*}
        \sum_{j'\leq j\leq\ell}\frac{A^{j'-1}}{(j'-1)!}\ll_A\sum_{j'\leq j\leq\ell}1=O_A(1).
    \end{equation*}
    Thus,~\eqref{v1veq2} and thereby~\eqref{v1veq} follows.
\end{proof}

Finally, we require some pointwise bounds for $r(m_j)$. Before giving these, we first recall standard bounds on the arithmetic functions $\omega(d)$ and $\varphi(d)$.
\begin{lemma}[{\cite[Th\'eor\`eme 11]{robin1983estimation} and \cite[Theorem 328]{hardy2008introduction}}]\label{omegaphilem}
    One has
    \begin{align}
        \omega(d)\ll\frac{\log d}{\log\log d},\label{omegabound}\\
        \varphi(d)\gg\frac{d}{\log\log d}.\label{phibound}
    \end{align}
\end{lemma}

\begin{lemma}\label{rmjlem}
    Suppose $k_1\ge \log\mcX$. Then, for any $j<\ell$ and $\varepsilon>0$, we have
    \begin{equation*}
        r(m_j)\ll_{B,\gamma,\varepsilon}\frac{\mcX}{(\log\mcX)^{B+1-\varepsilon}}.
    \end{equation*}
\end{lemma}
\begin{proof}
    By~\eqref{sievelbmain2}, we have
    \begin{equation*}
        r(m_j)\ll\gamma_1^{\omega(m_j)}\sup_{y\leq \mcX}\max_{(a,m_j)=1}\left|\pi(y;m_j,a)-\frac{\pi(y)}{\varphi(m_j)}\right|+\sqrt{\mcX}\log\mcX.
    \end{equation*}
    Applying Lemma \ref{PNTAPlem} and partial summation gives
    \begin{equation*}
        \sup_{y\leq \mcX}\max_{(a,m_j)=1}\left|\pi(y;m_j,a)-\frac{\pi(y)}{\varphi(m_j)}\right|\ll_B\frac{\mcX}{(\log\mcX)^{B+1}}.
    \end{equation*}
    Here, since $j<\ell$ we note that $m_j$ is not divisible by the exceptional modulus up to $Q_1(\mcX)=(\log\mcX)^B$, so $\mcX^{\beta_{m_j}}$ (appearing in \eqref{PNTAPeq}) is suitably bounded by Lemma~\ref{exceptionalzerolem}. To finish, we use the bound $m_j<k_1\leq(\log\mcX)^B$ as well as \eqref{omegabound} of Lemma~\ref{omegaphilem} to obtain
    \begin{equation*}
        \gamma_1^{\omega(m_j)}< \gamma^{\omega(k_1)}\ll_{B,\gamma,\varepsilon}(\log\mcX)^{\varepsilon}.\qedhere
    \end{equation*}
\end{proof}

\begin{lemma}\label{rmllem}
    If $k_1\geq\log\mcX$ then for any $\varepsilon>0$,
    \begin{equation*}
        r(m_{\ell})\ll_{B,\gamma,\varepsilon}\frac{\mcX}{(\log\mcX)^{2-\varepsilon}}.
    \end{equation*}
\end{lemma}
\begin{proof}
    We have $m_{\ell}=k_1\in[\log\mcX,(\log\mcX)^B]$ and so by \eqref{sievelbmain2},
    \begin{equation*}
        r(m_{\ell})\ll \gamma^{\omega(k_1)}\sup_{y\leq \mcX}\max_{(a,k_1)=1}\left|\pi(y;k_1,a)-\frac{\pi(y)}{\varphi(k_1)}\right|+\sqrt{\mcX}\log\mcX.
    \end{equation*}
    So, consider any $y\leq\mcX$ and $a\in\mathbb{Z}$ with $(a,k_1)=1$. We claim that
    \begin{equation*}
        \left|\pi(y;k_1,a)-\frac{\pi(y)}{\varphi(k_1)}\right|\ll_B \frac{\mcX\log\log\log\mcX}{(\log\mcX)^2}.
    \end{equation*}
    The case $y\leq\sqrt{\mcX}$ follows from the trivial bounds $\pi(y;k_1,a)\leq y$ and ${\pi(y)\leq y}$, so we assume $y\geq\sqrt{\mcX}$. Then, by the Brun--Titchmarsh theorem,
    \begin{equation}\label{brunboundk1}
        \pi(y;k_1,a)\ll_B\frac{y}{\varphi(k_1)\log y}\quad\text{and}\quad\frac{\pi(y)}{\varphi(k_1)}\ll_B\frac{y}{\varphi(k_1)\log y}.
    \end{equation}
    Using \eqref{brunboundk1} along with \eqref{phibound} of Lemma \ref{omegaphilem} gives
    \begin{equation*}
        \left|\pi(y;k_1,a)-\frac{\pi(y)}{\varphi(k_1)}\right|\ll_B\frac{y}{\varphi(k_1)\log y}\ll\frac{y\log\log k_1}{k_1\log y}\ll_B\frac{\mcX\log\log\log\mcX}{(\log\mcX)^2}
    \end{equation*}
    as claimed. Now, we apply \eqref{omegabound} of Lemma \ref{omegaphilem} to obtain
    \begin{equation*}
        \gamma^{\omega(k_1)}\ll_{B,\gamma,\varepsilon}(\log\mcX)^\varepsilon.
    \end{equation*}
    The lemma then follows after a suitable rescaling of $\varepsilon$.
\end{proof}

\subsection{Proof of Theorem \ref{lowerthm}}
We now combine all our results to prove Theorem \ref{lowerthm}.

\begin{proof}[Proof of Theorem \ref{lowerthm}]
    When $k_1\leq\log\mcX$, the proof follows readily as in the proof of Theorem \ref{upperthm}. That is, one simply applies \eqref{pismallgammaeq} of Proposition \ref{mainPNTAPprop}. Thus, we only consider the case $k_1>\log\mcX$. 

    From Lemma \ref{exinclem}, it suffices to bound each term on the right-hand side of \eqref{inclexcleq}. To this aim, we define
    \begin{equation*}
        D_j:=\frac{D}{m_j}\quad\text{and}\quad s_j:=\frac{\log D_j}{\log z}
    \end{equation*}
    for $j=0,\ldots,\ell$. Here we remark that $s_j>1$ for sufficiently large values of $\mcX$. To see this, we recall that $m_j<k_1\leq(\log\mcX)^B$ and $D=\sqrt{\mcX}/(\log\mcX)^B$ so that $\log D_j\sim \log D$ and $s_j\to s>1$ as $\mcX\to\infty$.
    
    Now, we further set
    \begin{equation*}
        r_k(d):=|\mcA_{kd}|-g(d)|\mcA_k|
    \end{equation*}
    and
    \begin{equation*}
        E_j:=
        \begin{cases}
            \sum_{d<D_j,\:(d,\overline{\PP}_{j+1})=1}\gamma^{\omega(d)}\mu^2(d)|r_{m_j}(d)|,&\text{if $j=0,\ldots,\ell-1$},\\
            \sum_{d<D_\ell,\:(d,\overline{\PP}_\ell)=1}\gamma^{\omega(d)}\mu^2(d)|r_{m_{\ell}}(d)|,&\text{if $j=\ell$},
        \end{cases}
    \end{equation*}
    with $\PP_j$ as defined in~\eqref{jdefs}.
    Applying the upper bound sieve \eqref{sieveubgen} gives,
    \begin{align}
        S(\mcA_{m_j},\PP_{j+1},z)&<|\mcA_{m_j}|V_{j+1}(z)(F(s_j)+\varepsilon_1(|\mcA_{m_j}|))+E_j\notag\\
        &=|\mcA_{m_j}|\left[V_{j+1}(z)+V_{j+1}(z)\left(F(s_j)-1+\varepsilon_1(|\mcA_{m_j}|)\right)\right]+E_j\notag\\
        &\leq|\mcA_{m_j}|\left[V_{j+1}(z)+V_{j+1}(z)\left(F(s_\ell)-1+\varepsilon_1(X)\right)\right]+E_j,\label{jub}
    \end{align}
    where in the last step we used that $s_\ell\leq s_j$ and $|\mcA_{m_j}|\leq|\mcA|=X$ so that ${F(s_j)\leq F(s_{\ell})}$ and $\varepsilon_1(|\mcA_{m_j}|)\leq \varepsilon_1(X)$. 
    Then, analogously applying the lower bound sieve \eqref{sievelbgen}:
    \begin{equation}
        S(\mcA_{m_j},\PP_{j+1},z)>|\mcA_{m_j}|\left[V_{j+1}(z)-V_{j+1}(z)\left(1-f(s_\ell)+\varepsilon_2(X)\right)\right]-E_j.\label{jlb}
    \end{equation}
    Similarly,
    \begin{align}
        S(\mcA_{m_\ell},\PP_{\ell},z)&<|\mcA_{m_\ell}|\left[V_{\ell}(z)+V_{\ell}(z)\left(F(s_\ell)-1+\varepsilon_1(X)\right)\right]+E_\ell\label{ellub}\\
        S(\mcA_{m_\ell},\PP_{\ell},z)&>|\mcA_{m_{\ell}}|\left[V_{\ell}(z)-V_{\ell}(z)\left(1-f(s_\ell)+\varepsilon_2(X)\right)\right]-E_{\ell}.\label{elllb}
    \end{align}
    Substituting \eqref{jub}--\eqref{elllb} into \eqref{inclexcleq} of Lemma \ref{exinclem} gives
    \begin{align}
        &S(\mcA,\PP,z)>\sum_{j=0}^{\ell-1}(-1)^j|\mcA_{m_j}|V_{j+1}(z)+(-1)^{\ell}|\mcA_{m_\ell}|V_\ell(z)\notag\\
        &\quad-|\mcA|V_1(z)\left[1-f(s_\ell)+\varepsilon_2(X)\right]\notag\\
        &\quad -\left(\sum_{j=1}^{\ell-1}|\mcA_{m_j}|V_{j+1}(z)+|\mcA_{m_\ell}|V_\ell(z)\right)\left(\max\left\{F(s_\ell)-1,1-f(s_\ell)\right\}+\max\left\{\varepsilon_1(X),\varepsilon_2(X)\right\}\right)\notag\\
        &\quad-\sum_{j=0}^\ell E_j.\label{bigeq}
    \end{align}
    We now simplify each line of \eqref{bigeq}. For the first line, we have by Lemma \ref{Vinductlem},
    \begin{align}
        &\sum_{j=0}^{\ell-1}(-1)^j|\mcA_{m_j}|V_{j+1}(z)+(-1)^{\ell}|\mcA_{m_\ell}|V_\ell(z)\notag\\
        &\qquad=\sum_{j=0}^{\ell-1}(-1)^j(g(m_j)|\mcA|+r(m_j))V_{j+1}(z)+(-1)^{\ell}(g(m_\ell)|\mcA|+r(m_\ell))V_\ell(z)\notag\\
        &\qquad=|\mcA|V(z)+\sum_{j=0}^{\ell-1}(-1)^{j}r(m_j)V_{j+1}(z)+(-1)^{\ell}r(m_{\ell})V_{\ell}(z)\notag\\
        &\qquad\geq XV(z)-V_{\ell}(z)\left(\sum_{j=0}^{\ell-1}|r(m_j)|+|r(m_{\ell})|)\right),\label{midpointvleq}
    \end{align}
    noting that $|\mcA|=X$. Here, by the condition \eqref{omegacon}, we have
    \begin{align}\label{VlVzeq}
        \frac{V_\ell (z)}{V(z)}=\prod_{i=1}^\ell(1-g(q_i))^{-1}\leq\prod_{\substack{p\leq k_1\\p\in\PP}}(1-g(p))^{-1}\ll_L(\log k_1)^{\kappa}\ll_{L,B}(\log\log\mcX)^{\kappa}. 
    \end{align}
    Substituting \eqref{VlVzeq} into \eqref{midpointvleq} and then applying Lemmas \ref{omegaphilem}, \ref{rmjlem} and \ref{rmllem} gives
    \begin{align}
        &\sum_{j=0}^{\ell-1}(-1)^j|\mcA_{m_j}|V_{j+1}(z)+(-1)^{\ell} |\mcA_{m_\ell}| V_\ell(z)\notag\\
        &\qquad\geq XV(z)+V(z)\cdot O_{L,B,\gamma,\varepsilon}\left((\log\log\mcX)^{\kappa}\left(\frac{\mcX(\log\log\mcX)}{(\log\mcX)^{B+1-\varepsilon}\log\log\log\mcX}+\frac{\mcX}{(\log\mcX)^{2-\varepsilon}}\right)\right)\notag\\
        &\qquad=XV(z)\left(1+O_{\kappa,L,B,\gamma,\varepsilon}\left(\frac{1}{(\log X)^{1-\varepsilon}}\right)\right),\label{firstbig}
    \end{align}
    where in the last step we have suitably rescaled $\varepsilon$, and used that $B>1$ and $\mcX\asymp X\log X$. This completes our analysis of the first line of \eqref{bigeq}.

    For the second line of \eqref{bigeq}, we first note that by \eqref{zseq}
    \begin{equation*}
        s_{\ell}=\frac{\log D/k_1}{\log z}=s-\frac{\log k_1}{\log z}\geq s-\delta,
    \end{equation*}
    with $\delta>0$ satisfying
    \begin{equation}\label{deltadef}
        \delta=O_{s,B}\left(\frac{\log\log \mcX}{\log \mcX}\right)=o_{s,B}(1).
    \end{equation}
    Thus, since $f(s)$ is non-decreasing,
    \begin{equation}\label{fsleq}
        f(s_{\ell})\geq f(s-\delta).
    \end{equation}
    Using \eqref{fsleq}, Lemma \ref{V1lem} and $|\mcA|=X$ gives
    \begin{align}\label{secondbig}
        &-|\mcA|V_1(z)\left[1-f(s_\ell)+\varepsilon_2(X)\right]\notag\\
        &\qquad\qquad\qquad\geq -XV(z)\left(1+O_A\left(\frac{1}{\log\log\mcX}\right)\right)\left[1-f(s-\delta)+\varepsilon_2(X)\right].
    \end{align}
    We now move onto the third line of \eqref{bigeq}. Here, we have
    \begin{align}
        &\sum_{j=1}^{\ell-1}|\mcA_{m_j}|V_{j+1}(z)+|\mcA_{m_\ell}|V_\ell(z)\notag\\
        &\qquad=X\left(\sum_{j=1}^{\ell-1}g(m_j)V_{j+1}(z)+g(m_{\ell})V_\ell(z)\right)+\sum_{j=1}^{\ell-1}r(m_j)V_{j+1}(z)+r(m_{\ell})V_{\ell}(z).\label{thirdlineeq1}
    \end{align}
    Now, by Lemma \ref{v1vlem},
    \begin{equation}\label{thirdlineeq2}
        \sum_{j=1}^{\ell-1}g(m_j)V_{j+1}(z)+g(m_{\ell})V_\ell(z)=V(z)\cdot O_A\left(\frac{1}{\log\log\mcX}\right).
    \end{equation}
    Then, as when bounding the first line of \eqref{bigeq} (see \eqref{firstbig}), we have
    \begin{equation}\label{thirdrmjeq}
        \sum_{j=1}^{\ell-1}r(m_j)V_{j+1}(z)+r(m_{\ell})V_{\ell}(z)\leq XV(z)\cdot O_{\kappa,L,B,\gamma,\varepsilon}\left(\frac{1}{(\log X)^{1-\varepsilon}}\right).
    \end{equation}
    Substituting \eqref{thirdlineeq2} and \eqref{thirdrmjeq} into \eqref{thirdlineeq1} gives that the third line of \eqref{bigeq} is bounded below by
    \begin{equation}\label{penultimatethird}
        -XV(z)\cdot O_{A,\kappa,L,B,\gamma}\left(\frac{1}{\log\log\mcX}\right)\cdot\left(\max\left\{F(s_\ell)-1,1-f(s_\ell)\right\}+\max\left\{\varepsilon_1(X),\varepsilon_2(X)\right\}\right).
    \end{equation}
    Since $\varepsilon_1(X),\varepsilon_2(X)\to 0$, the $\max\{\varepsilon_1(X),\varepsilon_2(X)\}$ term  in \eqref{penultimatethird} can be absorbed into the big-$O$ factor. Then, as $F(\cdot)>1$ is decreasing and $f(\cdot)<1$ is increasing, we have
    \begin{equation*}
        \max\{F(s_{\ell})-1,1-f(s_{\ell})\}\}\leq\max\{F(s-\delta)-1,1-f(s-\delta)\}=O_{s,B}(1),
    \end{equation*}
    recalling that $\delta=o_{s,B}(1)$. In particular, \eqref{penultimatethird} can be further bounded below by
    \begin{equation}\label{finalthird}
        -XV(z)\cdot O_{A,\kappa,L,B,\gamma,s}\left(\frac{1}{\log\log\mcX}\right).
    \end{equation}
    Finally we deal with the fourth line in \eqref{bigeq}. Here, we simply apply Proposition \ref{pimidkprop} and \eqref{omegabound} of Lemma \ref{omegaphilem} to obtain
    \begin{align}\label{fourthlinefinal}
        \sum_{j=0}^\ell E_j\leq\omega(k_1)\cdot O_{\gamma,B}\left(\frac{\mcX}{(\log\mcX)^{B_\gamma}}\right)&\ll_{B,\gamma}\frac{\mcX\log\log\mcX}{(\log\mcX)^{B_{\gamma}}\log\log\log\mcX}\notag\\
        &\ll\frac{X\log\log X}{(\log X)^{B_{\gamma}-1}\log\log\log X}
    \end{align}
    Substituting each of the results \eqref{firstbig}, \eqref{secondbig}, \eqref{finalthird} and \eqref{fourthlinefinal} into \eqref{bigeq} completes the proof of the theorem.
\end{proof}

\section{Applications: proofs of Theorems \ref{twingoldthm} and \ref{p2thm}}\label{appsect}
In this section, we prove Theorems \ref{twingoldthm} and \ref{p2thm}. These theorems are intended as routine examples of how to apply Theorems \ref{upperthm} and \ref{lowerthm}.

\subsection{Proof of Theorem \ref{twingoldthm}}
We begin with the proof of Theorem \ref{twingoldthm}, which gives an upper bound for the count of twin primes and Goldbach representations.
\begin{proof}[Proof of Theorem \ref{twingoldthm}]
    For $x\geq 3$ and even $n\geq 4$, we set
    \begin{align*}
        \mcA_1&=\{p+2:2<p\leq x\ \text{is}\ \text{prime}\}\quad\text{and}\quad\PP_1=\{p>2\ \text{prime}\},\\
        \mcA_2&=\{n-p:p\leq n,\:(p,n)=1\}\quad\text{and}\quad\PP_2=\{p\ \text{prime}:(p,n)=1\}.
    \end{align*}
    With these choices of sifting sets, we have that
    \begin{align}
        \Pi_2(x)&\leq\pi(z)+S(\mcA_1,\PP_1,z)\label{pi2id},\\
        G(n)&\leq\pi(z)+S(\mcA_2,\PP_2,z)+\omega(n),\label{gnid}
    \end{align}
    where $\Pi_2(x)$ and $G(n)$ are as defined in \eqref{pi2def} and \eqref{Gndef} respectively. Note here that in \eqref{gnid} the $\omega(n)=O(\log n)$ term comes from considering all primes $p$ with $p\mid n$, which are not included in the definitions of $\mcA_2$ and $\PP_2$. 
    
    Now, in what follows, we will only focus on bounding $\Pi_2(x)$, since the bound for $G(n)$ follows via an essentially identical argument using \eqref{gnid} in place of \eqref{pi2id}. To set up our sifting argument, we let $X=|\mcA_1|=\pi(x)-1$ and
    \begin{equation*}
        g(p)=\frac{1}{\varphi(p)}=\frac{1}{p-1}
    \end{equation*}
    for all $p\in\PP_1$ which extends to the multiplicative function $g(d)=1/\varphi(d)$ for all square-free $d$ with $(d,\overline{\PP_1})=1$. With this choice of $g(d)$, one has
    \begin{equation*}
        r(d)=|(\mcA_1)_d|-g(d)|\mcA_1|=\left|\pi(x;d,2)-\frac{\pi(x)}{\varphi(d)}\right|+O(1).
    \end{equation*}
    Moreover, via an application of Mertens' theorem,
    \begin{equation*}
        \prod_{\substack{z_1\leq p<z_2\\p\in\PP_1}}(1-g(p))^{-1}\leq\frac{\log z_2}{\log z_1}\left(1+\frac{L}{\log z_1}\right)
    \end{equation*}
    for some effectively computable constant $L>0$,
    and,
    \begin{align}
        V(z):=\prod_{\substack{2<p<z}}(1-g(p))&=\prod_{2<p<z}\left(1-\frac{1}{(p-1)^2}\right)\prod_{\substack{2<p<z}}\left(1-\frac{1}{p}\right)\notag\\
        &=\frac{2e^{-\gamma}}{\log z}\prod_{p>2} \left(1-\frac{1}{(p-1)^2}\right)\left(1+O\left(\frac{1}{\log z}\right)\right)\label{Vmertenexp},
    \end{align}
    with the big-$O$ term effective. Note that in \eqref{Vmertenexp} we have used that
    \begin{equation*}
        \prod_{p\geq z}\left(1-\frac{1}{(p-1)^2}\right)^{-1}=\prod_{p\geq z}\left(1+\frac{1}{p(p-2)}\right)\leq\prod_{n\geq z}\left(1+\frac{3}{n^2}\right)=1+O\left(\frac{1}{z}\right).\\
    \end{equation*}
    Thus, in the context of Theorem \ref{upperthm}, we have a sifting problem with $A=1$ and $\kappa=1$. We now use the Rosser--Iwaniec linear sieve upper bound \cite[Theorem 1]{iwaniec1980rosser}
    \begin{equation}\label{rosserupper}
        S(\mcA_1,\PP_1,z)< XV(z)\left(F(s)+O_L\left(\frac{1}{(\log X)^{1/3}}\right)\right)+\sum_{\substack{d\leq D\\(d,\overline{\PP_1})=1}}\mu^2(d)|r(d)|,
    \end{equation}
    where, for $1\leq s\leq 3$
    \begin{equation*}
        F(s)=\frac{2e^{\gamma}}{s},
    \end{equation*}
    and for $s>3$, $F(s)$ is defined by a differential-delay equation and satisfies the conditions of Theorem \ref{upperthm}, see \cite[\S 8.2]{halberstam1974sieve}. Here, we note that the $O_L$ term in \eqref{rosserupper} is effective and a fully explicit variant is given in \cite[Theorem 2.1]{BJV2025}. Thus, letting 
    \begin{equation*}
        D=\frac{\sqrt{x}}{(\log x)^{5}},
    \end{equation*}
    we can apply Theorem \ref{upperthm} with $A=1$, $\gamma=1$, $B=5$, $s=1$ and $V(z)$ given in \eqref{Vmertenexp} to obtain the effective upper bound
    \begin{align*}
        \Pi_2(x)&=\pi(D)+S(\mcA_1,\PP_1,D)\\
        &<\pi(\sqrt{x})+4C_2\frac{\pi(x)}{\log x}\left(1+O\left(\frac{1}{\log\log X}\right)\right)+O\left(\frac{\pi(x)}{(\log x)^2}\right).
    \end{align*}
    The desired bound \eqref{pi2gbound} then follows upon applying an effective version of the prime number theorem $\pi(x)\sim x/\log x$ (e.g.\ \cite[Theorem 1]{rosser1962approximate}).
\end{proof}
\begin{remark}
    It is also common to use Selberg's upper bound sieve \cite[Theorem 3.2]{halberstam1974sieve} to obtain an upper bound for $\Pi_2(x)$ or $G(n)$. Selberg's sieve would give a simpler $O_L$ term in \eqref{rosserupper} yet a slightly more complicated remainder term. Ultimately, using either sieve (Rosser-Iwaniec or Selberg) yields the same upper bound asymptotically, but the size of $X(\varepsilon)$ and $N(\varepsilon)$ in Theorem \ref{twingoldthm} would differ. 
\end{remark}

\subsection{Proof of Theorem \ref{p2thm}}
We now prove Theorem \ref{p2thm}, which gives an application of the effective lower bound sieve, Theorem \ref{lowerthm}. This application also gives an example of the case when $\gamma>1$ and $\kappa>1$, unlike our proof of Theorem \ref{twingoldthm}.

\begin{proof}[Proof of Theorem \ref{p2thm}]
    Let $n\geq 5$ such that $n\equiv 0,2$ (mod 6). We set
    \begin{align*}
        \mcA&=\{n-q^2:3<q\leq \sqrt{n}\ \text{is}\ \text{prime}\ \text{and}\ (q,n)=1\},\\
        \quad\PP&=\{p>3\ \text{prime}\ \text{and}\ (p,n)=1\}.
    \end{align*}
    We first confirm that any $a\in\mcA$ has no divisors in $\overline{\PP}$. To see this, consider some $a=n-q^2\in\mcA$. Then $(a,n)=1$ since $(q,n)=1$. In addition, $2\nmid a$ since ${n\equiv 0\pmod{2}}$ and $q^2\equiv 1\pmod{2}$. Similarly $3\nmid a$. All this is to say that $S(\mcA,\PP,z)$ gives a lower bound for the number of elements of $\mcA$ with no prime factor less than $z$. In particular, for the purposes of the proof, we wish to show
    \begin{equation*}
        S(\mcA,\PP,z)>0
    \end{equation*}
    for some $z>n^{1/18}$ and sufficiently large $n>N$, with $N$ effectively computable.

    With a view to defining a suitable multiplicative function $g(d)$ for sifting, we now analyse the quantity $|\mcA_p|$ for $p\in\PP$. Here we are looking for the number of primes $q\in\PP$ with $q\leq\sqrt{n}$ such that
    \begin{equation}\label{adcong}
        n\equiv q^2\pmod{p}.
    \end{equation}
    For $n$ to give a solution to \eqref{adcong} we require that $n$ is a quadratic residue modulo $p$. That is, there exists ${a_0\pmod{p}}$ with $a_0^2\equiv n$ (mod $p$). If this is the case, then there are two congruence classes of solutions to \eqref{adcong}, namely $a_0$ and $-a_0$. In particular,
    \begin{align}
        |\mcA_p|&=\#\{n\equiv q^2\ \text{(mod $p$)}:q\in\PP\}\notag\\
        &=\#\{n\equiv q^2\ \text{(mod $p$)}:q\leq\sqrt{n}\ \text{prime}\}+O(\log n)\label{Apexp}\\
        &=\#\{q\leq \sqrt{n}:q\equiv a_0\ \text{(mod $p$)}\ \text{or}\ q\equiv -a_0\ \text{(mod $p$)}\}+O(\log n)\notag\\
        &=\pi(\sqrt{n};p,a_0)+\pi(\sqrt{n};p,-a_0)+O(\log n)\label{PNTAPsplit}\\
        &\sim\frac{2}{\varphi(p)}\pi(\sqrt{n})\label{pntapquad}
    \end{align}
    where in \eqref{Apexp} we used that there are $O(\log n)$ unique prime divisors of $n$. By the Chinese remainder theorem, the asymptotic \eqref{pntapquad} extends to
    \begin{equation*}
        |\mcA_d|\sim\frac{2^{\omega(d)}}{\varphi(d)}\pi(\sqrt{n})
    \end{equation*}
    for any square-free $d$ with $(d,\overline{\mathcal{P}})=1$. Thus, for our sifting problem we set
    \begin{equation*}
        g(d)=
        \begin{cases}
            \frac{2^{\omega(d)}}{\varphi(d)},&\text{if $n$ is a quadratic residue mod $d$},\\
            0,&\text{otherwise}
        \end{cases}
    \end{equation*}
    so that $|\mcA_d|\sim g(d)|\mcA|$. More generally, the Chinese remainder theorem gives that the condition \eqref{sievelbmain2} holds with $\gamma_1=2$. We then have for ${z_2>z_1>3}$,
    \begin{align}\label{dim2eq}
        \prod_{\substack{z_1\leq p<z_2\\p\in\PP}}(1-g(p))^{-1}<\prod_{\substack{z_1\leq p<z_2}}\left(1-\frac{2}{p-1}\right)^{-1}&=\exp\left(-\sum_{\substack{z_1\leq p<z_2}}\log\left(1-\frac{2}{p-1}\right)\right)\notag\\
        &=\exp\left(\sum_{z_1\leq p<z_2}\left(\frac{2}{p}+O\left(\frac{1}{p^2}\right)\right)\right)\notag\\
        &=\exp\left(\sum_{z_1\leq p<z_2}\frac{2}{p}+O(1)\right).
    \end{align}
    Applying Mertens' theorem to \eqref{dim2eq} then gives
    \begin{equation*}
        \prod_{\substack{z_1\leq p<z_2\\p\in\PP}}(1-g(p))^{-1}<\left(\frac{\log z_2}{\log z_1}\right)^2\left(1+\frac{L}{\log z_1}\right),
    \end{equation*}
    for some computable constant $L>0$. This means that we can apply a sieve of dimension $\kappa=2$. Via analogous reasoning, we also see that
    \begin{equation}\label{vzbound18}
        V(z)=\prod_{\substack{p<z\\p\in\PP}}(1-g(p))\gg\frac{1}{(\log z)^2}.
    \end{equation}
    Now, the 2-dimensional sieve lower bound we choose to apply is due to Diamond, Halberstam and Richert, given by \cite[Theorem 9.1]{diamond2008higher}
    \begin{equation}\label{DHRlb}
        S(\mcA,\PP,z)>XV(z)\left\{f_{2}(s)+O_{L}\left(\frac{(\log\log D)^2}{(\log D)^{1/6}}\right)\right\}-R_{DHR}(D),
    \end{equation}
    where
    \begin{equation}
        R_{DHR}(D)=2\sum_{\substack{d<D\\(d,\overline{\PP})=1}}4^{\omega(d)}|r(d)|\leq\sum_{\substack{d<D\\(d,\overline{\PP})=1}}8^{\omega(d)}|r(d)|\label{rdgen},
    \end{equation}
    $f_{2}(s)$ is defined in \cite[Chapter 6.1]{diamond2008higher}, and the $O_L$ term in \eqref{DHRlb} is effective. Notably, $f_2(s)>0$ for $s\geq 4.267$ \cite[Table 17.1]{diamond2008higher}. 
    
    Therefore, with $X=|\mcA|\sim\pi(\sqrt{n})$ and $V(z)$ bounded as in \eqref{vzbound18}, we can apply Theorem \ref{lowerthm} with $A=2$, $\mcX=\sqrt{n}$, $B=266$, $\gamma=16$ and  $s=4.267$ to yield an effective bound
    \begin{equation*}
        S(\mcA,\PP,z)\gg\frac{\sqrt{n}}{(\log n)^3}>0
    \end{equation*}
    with 
    \begin{equation*}
        z=D^{1/s}=\frac{n^{1/17.068}}{(\log\sqrt{n})^{266/4.267}}.
    \end{equation*}
    In particular, $z>n^{1/18}$ for sufficiently large $n$, so that the theorem follows.
\end{proof}
\begin{remark}
    If one were to make the above proof explicit, then the value of $N$ for which $S(\mcA,\PP,z)>0$ for all $n>N$ would be very large. This is because of the value of $B=266$ which would require one to use a very good bound for the error term in the prime number theorem for arithmetic progressions (see Lemma \ref{PNTAPlem}). To be able to use a lower value of $B$, one should use a sieve that requires as small a value of $\gamma$ as possible. For example, if the 2-dimensional Rosser--Iwaniec sieve \cite[Theorem~1]{iwaniec1980rosser} was used instead of \eqref{DHRlb}, then one could take $\gamma=2$ and $B=14$. However, this would come at the cost of slightly increasing the number of prime factors of $n-q^2$.
\end{remark}

\section*{Acknowledgements}
The author is grateful to the Max Planck Institute for Mathematics in Bonn
for its hospitality and financial support. The author thanks the anonymous referees for their useful comments. The author also thanks Bryce Kerr for his comments on a preliminary version of this paper.

\printbibliography

\end{document}